\documentclass[12pt]{amsart}
\usepackage{amsmath,amssymb}

\newcommand{\Ch}{\mathrm{Ch}}
\newcommand{\CuP}{\mathop{\mbox{\large $\cup$}}}

\newcommand{\ssum}{\mathop{\textstyle \sum}}
\newcommand{\mcl}{\mathrm{cl}}
\newcommand{\A}{\mathcal{A}}
\newcommand{\B}{\mathcal{B}}
\newcommand{\rr}{\mathrm{r}}

\newtheorem{theorem}{Theorem}[section]
\newtheorem{lemma}[theorem]{Lemma}

\newtheorem{cor}[theorem]{Corollary}

\newtheorem{quest}[theorem]{Question}

\theoremstyle{definition}
\newtheorem{definition}[theorem]{Definition}
\newtheorem{example}[theorem]{Example}

\theoremstyle{remark}

\numberwithin{equation}{section}

\begin{document}
\title[]
{Multiplicatively spectrum-preserving and norm-preserving 
maps between invertible groups of 
commutative Banach algebras}

\author[O.~Hatori]{Osamu~Hatori}
\address{Department of Mathematics, Faculty of Science, 
Niigata University, Niigata 950-2181 Japan}
\curraddr{}
\email{hatori@math.sc.niigata-u.ac.jp}
\thanks{The authors were partly 
supported by the Grants-in-Aid for Scientific 
Research, The 
Ministry of Education, Science, Sports and Culture, Japan.}

\author[T.~Miura]{Takeshi Miura}
\address{Department of Basic Technology,
Applied Mathematics and Phisics, 
Yamagata University, Yonezawa 992-8510 Japan}
\curraddr{}
\email{miura@yz.yamagata-u.ac.jp}

\author[H.~Takagi]{Hiroyuki Takagi}
\address{Department of Mathematical Sciences,
Faculty of Science, Shinshu University, Matsumoto 390-8621 Japan}
\curraddr{}
\email{takagi@math.shinshu-u.ac.jp}

\date{}

\begin{abstract}
Let $A$ and $B$ be unital semisimple commutative Banach algebras and 
$T$ a map from the invertible group $A^{-1}$ onto $B^{-1}$. 
Linearity and multiplicativity of the map are not assumed. 
We consider the hypotheses on $T$: 
(1) $\sigma (TfTg)=\sigma (fg)$; 
(2) $\sigma_{\pi}(TfTg-\alpha)\cap \sigma_{\pi}(fg-\alpha )\ne \emptyset$; 
(3) $\rr (TfTg-\alpha )=\rr(fg-\alpha )$ hold for 
some non-zero complex number $\alpha$ and for every $f, g\in A^{-1}$, where 
$\sigma (\cdot )$ (resp. $\sigma_{\pi}(\cdot )$) denotes 
the (resp. peripheral) spectrum and $\rr(\cdot)$ denotes the spectral 
radius. Under each of the hypotheses we show representations 
for $T$ and under additional 
assumptions we show that $T$ is extended to an 
algebra isomorphism. In particular, 
if $T$ is a surjective group homomorphism 
such that $T$ preserves the spectrum or $T$ is a surjective isometry 
with respect to the spectral radius, then 
$T$ is extended to an algebra isomorphism.
Similar results holds for maps from $A$ onto $B$.
\end{abstract}

\maketitle

\section{Introduction}
Recently spectrum-preserving maps on Banach algebras 
which are not assumed to be linear are studied by several authors 
including
\cite{hmt1,hmt2,h,l,lt,m1,ps,ra,rr1,rr2}. 
In this paper we mainly consider 
maps which are defined on the invertible groups of commutative Banach 
algebras. 

For unital Banach algebras the groups of all invertible elements 
can be 
isomorphic 
as groups  to each other 
while these Banach algebras are not algebraically 
isomorphic to each other. Let  $C([0,1])$ be the Banach algebra of 
all continuous complex-valued functions on the closed unit interval. 
Then the group $C([0,1])^{-1}$ of all invertible elements of $C([0,1])$ 
is isomorphic as a group to the group of all non-zero complex numbers, 
the invertible group of the one dimensional Banach algebra ${\mathbb C}$; 
the complex number field. We show a proof for a convenience. Let 
$\{z_{\alpha}\}$ be a Hamel basis for the linear space ${\mathbb C}$ 
over the rational number field. We may assume that $\pi\in \{z_{\alpha}\}$. 
Since the cardinal number of $C([0,1])$ is the continuum ${\mathfrak c}$, 
the cardinal number of a Hamel basis for $C([0,1])$ over the rational 
number field is also the continuum. It follows that there is a 
linear isomorphism $S$ (over the rational number field) from ${\mathbb C}$ 
onto $C([0,1])$ such that $S(\pi)=\pi$. Put the map $T$ from $\exp{\mathbb C}$ 
onto $\exp(C([0,1]))$ as $T(\exp(z))=\exp(S(z))$: $T$ is well-defined 
since $S(\pi)=\pi$. By the definition $T$ is 
multiplicative. Since the equality 
$\exp (C([0,1]))=C([0,1])^{-1}$ holds (by a theorem of Arens and Royden 
\cite[Corollary III. 7. 4]{ga}, for example), we see that
$T$ is a group isomorphism from $\{z\in {\mathbb C}:
z\ne0\}$ onto $C([0,1])^{-1}$. 
In the same way there are various pairs of 
unital commutative Banach algebras which are not isomorphic as algebras 
while the invertible groups are isomorphic as groups. 
An interesting example of a group isomorphism between 
invertible groups of two non-isomorphic commutative $C^*$-algebras 
which is also a homeomorphism with respect to the relative topology 
induced by the norms on the algebras is presented in the monograph of 
\. Zelazko \cite[Remark 1. 7. 8]{ze}.

In spite of the above Hochwald \cite{h} proved that if a group homomorphism 
from the group of all invertible matrices in 
 $M_n$ of all $n \times n$ matrices into itself preserves 
the spectrum, then it is extended to an algebra automorphism on 
$M_n$. 
We show a type of a theorem of Hochwald for the case of unital semisimple 
commutative Banach algebras (cf. Corollary \ref{commutative}). 
We also show that a group homomorphism between the invertible 
groups of unital semisimple Banach algebra is extended to a real-algebra 
isomorphisms between underlying algebras if $T$ is isometric with respect to 
the spectral radius (cf. Corollary \ref{cor3.3}).
In this paper we consider not only group homomorphisms with additional 
topological properties between 
invertible groups but also multiplicatively (resp. peripheral) 
spectrum-preserving maps and norm-preserving maps. We say that a map $T$ 
between Banach algebras 
is 
multiplicatively spectrum-preserving if 
\[
\sigma(TfTg)=\sigma (fg)
\]
holds for every pair $f$ and $g$ in the domain of $T$, where $\sigma (\cdot)$ 
denotes the spectrum. The peripheral spectrum $\{z\in \sigma (f):
|z|=\rr(f)\}$ is denoted by 
$\sigma _{\pi}(f)$, where $\rr (f)$ is the spectral radius for $f$. 
We say that $T$ is multiplicatively peripheral spectrum-preserving 
if 
\[
\sigma_{\pi}(TfTg)=\sigma _{\pi}(fg)
\]
holds for every pair $f$ and $g$ in the domain of $T$, and
$T$ is multiplicatively norm-preserving if 
\[
\|TfTg\|=\|fg\|
\]
holds for every pair $f$ and $g$ in the domain of $T$ 
for a certain norm $\|\cdot \|$ including the spectral radius. 
The study of multiplicatively 
spectrum-preserving maps between Banach algebras was initiated by 
Moln\'ar \cite{m1} and he characterized algebra isomorphisms in terms of 
multiplicative spectrum-preservingness. Rao and Roy \cite{rr1} and Hatori, 
Miura and Takagi \cite{hmt1} generalized the theorem of Moln\'ar 
for uniform algebras. Hatori, Miura and Takagi \cite{hmt2}
also generalizes for unital semisimple commutative Banach algebras. 
Luttman and Tonev \cite{lt} introduced multiplicatively peripheral 
spectrum-preserving maps and generalizes results of Rao and Roy \cite{rr1} and 
Hatori, Miura and Takagi \cite{hmt1} in the case of uniform algebras.
Lambert, Luttman and Tonev \cite{l} considered the maps with much weaker 
conditions such as mutiplicatively norm-preservingsess or 
weakly peripherally-multiplicativity.

After some preliminaries in the next section, we study  multiplicatively 
norm-preserving maps between the invertible groups of 
commutative Banach algebras in section three. We show that 
commutative $C^*$-algebras are algebraically isomorphic if 
there exists a multiplicatively norm-preserving surjection, in 
particular, a norm-preserving group isomorphism between the invertible 
groups of the algebras. 

In section four we consider multiplicatively 
spectrum-preserving maps, in particular, peripheral spectrum-preserving 
group isomorphisms between invertible groups of uniform algebras
and show that they are extended to algebra isomorphisms 
between underling  algebras.

In section five we consider non-symmetric multiplicatively 
norm-preserving maps and we consider non-symmetric multiplicatively 
(resp. peripheral) spectrum-preserving maps in section six. 
We say that a map $T$ is non-symmetric multiplicatively norm-preserving 
if 
\[
\|TfTg-\alpha\|=\|fg-\alpha\|
\]
holds for every pair $f$ and $g$ in the domain of $T$ and some 
non-zero complex number $\alpha$.
We say that $T$ is non-symmetric multiplicatively spectrum-preserving 
if 
\[
\sigma (TfTg-\alpha )=\sigma (fg-\alpha )
\]
holds for every pair $f$ and $g$ in the domain of $T$ and some 
non-zero complex number $\alpha$. We say that $T$ is non-symmetric 
multiplicatively peripheral spectrum-preserving 
if 
\[
\sigma_{\pi}(TfTg-\alpha )=\sigma_{\pi}(fg-\alpha )
\]
holds for every pair $f$ and $g$ in the domain of $T$ and some 
non-zero complex number $\alpha$.

In the last section 
we consider maps between commutative Banach algebras which are 
multiplicatively norm-preserving, and show a generalization of a 
theorem of Luttman and Tonev \cite{lt}. We also study non-symmetric 
multiplicatively norm-preserving maps between commutative Banach algebras.

\section{preliminaries}
Suppose that $A$ is a unital commutative Banach algebra. 
We call the group of all invertible elements in $A$ the invertible 
group of $A$. The invertible group of $A$ is denoted by $A^{-1}$. 
The spectrum of $f\in A$ is denoted by $\sigma (f)$. The peripheral 
spectrum $\sigma_{\pi}(f)$ is the set $\{z\in \sigma (f):|z|=
r(f)\}$, where $r(f)$ denotes the spectral radius 
of $f\in A$. The maximal ideal 
space of $A$ is a compact Hausdorff space and 
is denoted by $M_A$. Then by Gelfand theory, 
the equality 
$\sigma (f)=\hat f(M_A)$ holds, 
where $\hat f$ is the Gelfand transform of $f$.
For a compact Hausdorff space $X$, the algebra of all complex-valued 
continuous functions on $X$ is denoted by $C(X)$.
We denote $\mcl A$ the uniform closure of the Gelfand transform $\hat A$ 
of 
$A$ in $C(M_A)$. 
Then $\mcl A$ is a uniform algebra on $M_A$. 
In this paper we denote the Gelfand transform of $f$ in $A$ also by 
$f$; omitting $\hat{\cdot}$, if $A$ is semisimple.

Let ${\mathcal A}$ be a uniform algebra on a compact Hausdorff space $X$, 
that is, ${\mathcal A}$ is a uniformly closed subalgebra of $C(X)$ 
which contains constants and separates the points of $X$.
For a subset $K$ of $X$ the supremum norm on $K$ is denoted by 
$\|g\|_{\infty (K)}$ for $g\in \A$. 
A function $f\in {\mathcal A}$ is said to be a peaking function if
$\sigma_{\pi}(f)=\{1\}$. 
Since the spectral radius and the supremum norm 
on $X$ coincide for uniform algebras, we see by a simple calculation 
that  $\sigma_{\pi}(g)=
\{z\in g(X):|z|=\|g\|_{\infty (X)}\}$ for every $g\in \A$. 
We denote the set of all peaking function by 
$P_{\A}$ and $P_{\A}(x)=\{f\in P_{\A}:f(x)=1\}$ for $x\in X$. 
We also denote 
\[
P_{\A}^0=\{f\in P_{\A}:0\not\in \sigma (f)\}
\]
and 
\[
P_{\A}^0(x)=\{f\in P_{\A}(x):0\not\in \sigma (f)\}
\]
for $x\in X$.
For a closed subset $K$ of $X$, we say that $K$ is a peak set 
if there is a peaking function $f\in \A$ such that $K=f^{-1}(1)$. 
If a peak set  is a singleton, then the unique 
element of the set is said to be a 
peak point.
An intersection of peak sets is said to be a $p$-set. If a $p$-set is a 
singleton, then the unique element of the set is said to be a $p$-point. 
The Choquet boundary for ${\mathcal A}$ is denoted by $\Ch ({\mathcal A})$. 
Note that $\Ch (\A)$ consists of all $p$-points. 
Note that for every $f\in \A$, $\sigma_{\pi}(f)\subset 
f(\Ch (\A))$ holds. 
(Suppose that $\alpha \in \sigma_{\pi} (f)$. If $\alpha=0$, then 
$\|f\|_{\infty (X)}=0$, so that the inclusion holds. 
If $\alpha \ne 0$, then by a simple calculation 
we see that $f^{-1}(\alpha )$ is a peak set for $\A$, so that 
$\Ch (\A)\cap f^{-1}(\alpha )\ne \emptyset$ by Corollary 2.4.6 in 
\cite{B}.)
See \cite{B,ga} for theory of uniform algebras.

The following is a version of a theorem of 
Bishop (cf. \cite[Theorem 2.4.1]{B}) and it is a generalization of 
Corollary 1 of \cite{l}.

\begin{lemma}\label{bishop}
Let $\A$ be a uniform algebra on a compact Hausdorff space $X$.
Let $f \in \A$ and $x_0 \in \Ch (\A)$.
If $f(x_0 )\ne 0$, 
then there exists a $u \in P_{\A}^0({x_0})$ 
such that $\sigma_{\pi}(fu)=\{f(x_0)\}$.
\end{lemma}

\begin{proof}
It is enough to show that there is a $u \in P_{\A}^0({x_0})$ 
such that $(1/f(x_0)) fu \in P_{\A}(x_0)$.
Put $\lambda = f(x_0)$. 
Put $F_0 = \bigl\{ \, x \in X \, : \, 
|f(x) - \lambda | \geq | \lambda |/2 \, \bigr\}$ and 
$$
 F_n = \left\{ \, x \in X \, : \, 
 \frac{| \lambda |}{2^{n+1}} \leq | f(x) - \lambda | 
 \leq \frac{| \lambda |}{2^n} 
              \right\} \quad ( n=1, 2, \ldots ) .
$$
Then $F_0 , F_1, \cdots , F_n , \cdots$ are  closed subsets of 
$X$ which 
do not contain $x_0$. 
Since $x_0$ is a $p$-point,
there exists a sequence of peaking functions 
$\{v_n\}_{n=0}^{\infty}$ such that $v_n(x_0)=1$ and $|v_n|<1/2$ on 
$F_n$ holds for every non-negative integer $n$. 
Let $\bar D$ be the closed unit 
disk $\{z\in {\mathbb C}:|z|\le 1\}$ and $\bar{\Omega}=
\{z\in \bar D|\mathrm{Re}z\ge0\}$. Put 
$\pi (z) = i\frac{\sqrt{(z+i)/(iz+1)}-i}{\sqrt{(z+i)/(iz+1)}+i}$. 
Then $\pi$ is a homeomorphism from $\bar D$ onto $\bar{\Omega}$ 
such that $\pi (1)=1$, $\pi (-1)=0$ and $\pi$ is analytic on 
the open unit disk $D$ onto the interior $\Omega$ of $\bar{\Omega}$. 
For every positive real number $\varepsilon$ there exists a 
positive real number $\delta_{\varepsilon}$ which satisfies that 
$|\pi (z)|<\varepsilon$ holds for every $z\in \bar D$ with 
$|z+1|<\delta_{\varepsilon}$. For every positive real number $\delta$ 
there exists a M\"obius transformation $\phi_{\delta}$ from 
$\bar D$ onto $\bar D$ such that $\phi_{\delta}(1)=1$ and 
$|\phi_{\delta}(z)+1|<\delta$ holds for every complex number $z$ with 
$|z|<1/2$. Put 
$u_0=\pi\circ\phi_{\delta_{|f(x_0)|/\|f\|_{\infty (X)}}}\circ v_0$. 
For a positive integer $n$, put 
$u_n=\pi\circ\phi_{\delta_{1/(2^n+1)}}\circ v_n$. Note that 
$\pi\circ\phi_{\varepsilon}$ is approximated by analytic polynomials on 
$\bar D$ since it is continuous on $\bar D$ and analytic on 
$D$. Thus $u_n \in \A$ for every non-negative integer $n$. 
We also see that for every non-negative integer $n$ 
$\mathrm{Re}u_n>0$ on $X$ since $v_n(X)\subset D\cup \{1\}$. 
By the definition of $\phi_{\varepsilon}$ we also see that 
$|u_0|<|\lambda| /\|f\|_{\infty (X)}$ on $F_0$ and 
$|u_n|<1/(2^n+1)$ on $F_n$ for every positive integer $n$. 

Now put 
$$
  u = u_0  \ssum _{k=1} ^{\infty} \frac{u_k }{2^k}.
$$
The above series is majorized by the convergent series $\ssum \frac{1}{2^k}$, 
so $u$ is in $\A$. Since $\mathrm{Re}u >0$ on $X$, we see that 
$u \in \exp \A\subset \A^{-1}$. 
Moreover, $u$ is easily seen to be a function in 
$P_{\A}^0({x_0})$.

Put $g = (1/\lambda ) fu$. Then this $g$ is a desired function. 
Choose an arbitrary $x \in X$. 
If $x \in F_0$, then we have
\begin{multline*}
|g(x)| = \frac{1}{| \lambda |} |f(x)| \, 
 |u_0 (x)| \ssum _{k=1} ^{\infty} \frac{|u_k (x)|}{2^k}\\
   < \frac{1}{| \lambda |} \| f \|_{\infty (X)} 
   \frac{| \lambda |}{\| f \|_{\infty(X)}} 
   \ssum _{k=1} ^{\infty} \frac{1}{2^k} = 1 .
\end{multline*}
If $x \in F_n$ for some positive integer $n$, then
\begin{equation*} \begin{split}
 |g(x)| & = \frac{1}{| \lambda |} |f(x)| \, |u_0 (x)| 
      \left( \frac{|u_n (x)|}{2^n} + \ssum _{k \neq n} 
      \frac{| u_k (x) |}{2^k} \right) \\
  & \leq \frac{1}{| \lambda |} \left( |f(x) - \lambda | + 
  | \lambda | \right) 
          \left( \frac{|u_n (x)|}{2^n} + \ssum _{k \neq n} 
          \frac{1}{2^k} \right) \\
  & < \frac{1}{| \lambda |} \left( \frac{| \lambda |}{2^n} + 
  | \lambda | \right) 
          \left( \frac{1}{2^n} \frac{1}{2^n + 1} + 
          1-\frac{1}{2^n} \right) =1 .
\end{split} \end{equation*}
If $x \in X \setminus \CuP _{n=0} ^{\infty} F_n$, 
then $f(x) = \lambda$ and so $g(x) = u(x)$.
Thus we have that $g(X) \subset D \cup \{ 1 \}$ and 
$g(x_0 ) = u( x_0 ) = 1$, so the proof is completed.
\end{proof}

\section{Multiplicatively norm-preserving maps between invertible groups}
Multiplicatively norm-preserving maps on uniform algebras are recently 
studied by Lambert, Luttman and Tonev \cite{l}. We show the following, 
which is a generalization of Theorem 1 in \cite{l}. 

\begin{theorem}\label{0korovkin}
Let $\A$ and $\B$ be uniform algebras on compact Hausdorff spaces $X$ and 
$Y$ respectively. 
Let $T$ be a map from $\A^{-1}$ onto $\B^{-1}$. 
Suppose that 
\[
\|TfTg\|_{\infty(Y)}=\|fg\|_{\infty(X)}
\]
holds for every $f, g\in \A^{-1}$.
Then there exists a homeomorphism $\phi$ from $\Ch (\B)$ onto $\Ch (\A)$ 
such that $|Tf(y)|=|f(\phi (y))|$ holds for every $f\in \A^{-1}$ and 
$y\in \Ch (\B)$.
\end{theorem}

Note that $T$ need not be injective. 

\begin{proof}
Let $y\in \Ch (\B)$ and put
\begin{multline*}
L_y=\{x\in X:\text{$|f(x)|=1$ for every}\\
\text{ $f\in \A^{-1}$ with 
$|Tf(y)|=1=\|Tf\|_{\infty(Y)}$}\}.
\end{multline*}
We show that $L_y \ne \emptyset$. Let $f_1, \dots, f_n\in \A^{-1}$ 
such that $|Tf_j(y)|=1=\|Tf_j\|_{\infty(Y)}$ for 
$j=1,2,\dots,n$. 
We show that 
\[
\cap_{j=1}^{n}|f_j|^{-1}(1)\ne \emptyset.
\]
By $\|T1T1\|_{\infty(Y)}=\|1^2\|_{\infty(X)}=1$ we have 
that $|T1|\le 1$ on $\Ch (\B)$. 
We show that $|T1|=1$ on $\Ch (\B)$. 
Suppose that $|T1(y)|<1$ for some $y\in \Ch (\B)$. Then by Lemma 
\ref{bishop} 
there exists an $H\in P_{\B}^0(y)$ such that 
$\|T1H\|_{\infty(Y)}=|T1(y)|$. Choose an $h\in \A^{-1}$ with $Th=H$. Then 
we have that $\|h\|_{\infty(X)}=1$ since
$\|h^2\|_{\infty(X)}=\|H^2\|_{\infty(Y)}=1$ and $\|\cdot\|_{\infty(\cdot)}$ 
is a uniform norm. 
It follows that $\|T1H\|_{\infty(Y)}=\|h\|_{\infty(X)}=1$, 
which is a contradiction. Thus 
we see that $|T1|=1$ on $\Ch (\B)$. 
Thus we see that 
\[
\|f\|_{\infty(X)}=\|1f\|_{\infty(X)}=\|T1Tf\|_{\infty(Y)}=\|Tf\|_{\infty(Y)}
\]
holds for every 
$f\in \A^{-1}$. 
Put $F=\prod_{j=1}^nTf_j$. Then $|F(y)|=1=\|F\|_{\infty(Y)}$. Since $T$ is 
a surjection we can choose  $f\in \A^{-1}$ such 
that $Tf=F$. Then there exists $x_0\in X$ with 
\[
|f(x_0)|=\|f\|_{\infty(X)}=\|F\|_{\infty(Y)}=1
\]
since $\|f\|_{\infty (X)}=\|Tf\|_{\infty (Y)}
=\|F\|_{\infty (Y)}=1$. Since $Tf=\prod_{j=1}^nTf_j$ and 
$\|Tf_j\|_{\infty(Y)}=1=\|Tf\|_{\infty(Y)}$, we have $|Tf|^{-1}(1)\subset 
|Tf_j|^{-1}(1)$, so 
\[
|f|^{-1}(1)\cap \Ch (\A) \subset |f_j|^{-1}(1).
\]
(Suppose that 
there exists $x_1\in \Ch (\A)$ such that $|f(x_1)|=1$ and 
$|f_j(x_1)|\ne 1$. Then $|f_j(x_1)|<1$ for $\|f_j\|_{\infty(X)}=1$. 
Then there exists a $u\in P_{\A}^0(x_1)$ such that 
$|uf_j|<1$ on $X$, so that $1>\|uf_j\|_{\infty(X)}=
\|TuTf_j\|_{\infty(Y)}$. On the other 
hand, since $1=\|uf\|_{\infty(X)}=\|TuTf\|_{\infty(Y)}$, 
there exists a $y_1\in Y$ such that 
$|Tu(y_1)Tf(y_1)|=1$, so that 
\[
|Tu(y_1)|=1=|Tf(y_1)|
\]
holds since $\|Tu\|_{\infty (Y)}=
\|u\|_{\infty (X)}$ and $\|Tf\|_{\infty (Y)}=1$. 
Since $|Tf|^{-1}(1)\subset |Tf_j|^{-1}(1)$ 
we see that 
$|Tu(y_1)Tf_j(y_1)|=1$, so $\|TuTf_j\|_{\infty(Y)}=1$, 
which is a contradiction.) 
Thus we see that 
\[
\cap_{j=1}^n|f_j|^{-1}(1)\supset \left(|f|^{-1}(1)\right)\cap
\Ch (\A).
\]
Note that $\left(|f|^{-1}(1)\right)\cap
\Ch (\A)\ne \emptyset$. (Since $\|f\|_{\infty(X)}=1$, 
there is an $x\in \Ch (\A)$ with 
$|f(x)|=1$. Then $f^{-1}(f(x))$ is a peak set since 
$(1+\overline{f(x)}f)/2$ peaks on $f^{-1}(f(x))$. Then 
\[
\emptyset \ne \Ch (\A)
\cap f^{-1}(f(x))\subset \Ch (\A)\cap |f|^{-1}(1).)
\]
Then 
$\cap_{j=1}^n|f_j|^{-1}(1)\ne \emptyset$. By the finite intersection property 
we see that $L_y\ne \emptyset$. 

Next we show that $L_y$ is a singleton and the unique element in $L_y$ 
is an element in $\Ch(\A)$. 
Let $x_2\in L_y$ and $f\in \A^{-1}$ such that $|Tf(y)|=1=\|Tf\|_{\infty(Y)}$. 
Then we have $|f(x_2)|=1$ by the definition of $L_y$.
Put $\tilde f= (\overline{f(x_2)}f+1)/2$. Then $\tilde f$ is a peak 
function such that $\tilde f(x_2)=1$. We see that
$\tilde f^{-1}(1)\subset |f|^{-1}(1)$, so 
\[
x_2\in \cap_f \tilde f^{-1}(1) \subset L_y.
\]
Since $\cap_f\tilde f^{-1}(1)$ is a $p$-set, there is an 
\[
x_0
\in \left(\cap_f\tilde f^{-1}(1)\right)\cap \Ch (\A).
\]
We show that 
$L_y=\{x_0\}$. Suppose that $x_3\in L_y\setminus\{x_0\}$. Then 
there exists a $u \in P_{\A}^0(x_0)$ such that $|u(x_3)|<1$. 
Since $\|Tu\|_{\infty(Y)}=1$ 
we see that $|Tu(y)|<1$ by the definition of $L_y$. 
(Suppose that $|Tu(y)|=1$. Then $|u|=1$ on $L_y$, which is a contradiction. 
So $|Tu(y)|\ne 1$, and since $\|Tu\|_{\infty (Y)}=1$, 
we have that $|Tu(y)|<1$.) 
Then there exists a $F'\in P_{\B}^0(y)$ such that 
$\|F'Tu\|_{\infty (Y)}<1$. Since $T$ is a surjection, there is an 
$f'\in \A^{-1}$ such that $Tf'=F'$. 
Then we have that 
\[
1>\|Tf'Tu\|_{\infty (Y)}=\|f'u\|_{\infty (X)}.
\]
Since
\[
Tf'(y)=F'(y)=1 
=\|F'\|_{\infty (Y)}=\|Tf'\|_{\infty (Y)}, 
\]
we have $|f'|=1$ on $L_y$, so 
we see that
$|f'(x_0)u(x_0)|=1$, so $\|f'u\|_{\infty (X)}=1$, which is a contradiction. 
We see that $L_y=\{x_0\}$ and $x_0\in \Ch (\A)$.

Put a function $\phi$ from $\Ch (\B)$ into $\Ch (\A)$ by 
$\phi(y)=x_0$, the unique element in $L_y$. We show that 
$|Tf(y)|=|f(\phi (y))|$ holds for every $f\in \A^{-1}$ and $y\in \Ch (\B)$. 
For the case where $f\in \A^{-1}$ satisfies that $|Tf(y)|=1=
\|Tf\|_{\infty (Y)}$, we see that 
$|Tf(y)|=|f(\phi (y))|$ holds for every $y\in \Ch (\B)$ by the definition 
of $\phi$. Let $f$ be an arbitrary function in $\A^{-1}$. 
Since $f(\phi(y))\ne 0$, there exists an $h\in P_{\A}^0(\phi (y))$ such that 
$\sigma_{\pi}(fh)=\{f(\phi (y))\}$ by Lemma \ref{bishop}, so 
we see that 
\[
|f(\phi (y))|=\|fh\|_{\infty(X)}=\|TfTh\|_{\infty(Y)}\ge
|Tf(y)Th(y)|.
\]
We see that $|Th(y)|=1$. (Suppose not. Then $|Th(y)|<1$ 
since $\|h\|_{\infty(X)}=\|Th\|_{\infty(Y)}=1$. 
Then there is an $H'\in P_B^0(y)$ such that 
$\|ThH'\|_{\infty(Y)}<1$. Choose an $h'\in \A^{-1}$ with $Th'=H'$. Then 
$|h'(\phi (y))|=1$ by the definition 
of $L_y$. We see that 
\[
\|ThH'\|_{\infty(Y)}=\|hh'\|_{\infty(X)}\ge |h(\phi (y))h'(\phi (y))|=1,
\]
which is a contradiction.) It follows that
$|f(\phi (y))|\ge |Tf(y)|$. 
On the other hand there exists an $H''\in P_{\B}^0(y)$ such that 
$\sigma_{\phi}(TfH'')=\{Tf(y)\}$. Choose an $h''\in \A^{-1}$ with 
$Th''=H''$. Since $Th''(y)=1=\|Th''\|$, we have 
$|h''(\phi (y))|=1$ by the definition of $L_y$. 
Thus 
\begin{multline}
|Tf(y)|=\|TfH''\|_{\infty(Y)}=\|fh''\|_{\infty(X)}\\
\ge 
|f(\phi (y))h''(\phi(y))|=|f(\phi (y))|.
\end{multline}
It follows that $|Tf(y)|=|f(\phi (y))|$. 

Next we show that $\phi$ is continuous. Let $y\in \Ch (\B)$. 
Suppose that $\{y_{\alpha}\}$ 
is a net in $\Ch (\B)$ which converges to $y$. Let $f\in \A^{-1}$. 
Since $Tf(y_{\alpha}) \to Tf(y)$, and $|Tf(y_{\alpha})|=|f(\phi(y_
{\alpha}))|$ and $|Tf(y)|=|f(\phi (y))|$, we see that
$|f(\phi(y_{\alpha})| \to |f(\phi (y))|$. By the Alexandroff theorem 
the original topology on $X$ coincides with the weak topology on $X$ which is 
induced by the family $\{|f|:f\in \A^{-1}\}$. It follows that 
$\phi(y_{\alpha})\to \phi(y)$. We see that $\phi$ is a continuous map 
from $\Ch (\B)$ into $\Ch (A)$. 

We show that $\phi$ is a homeomorphism. For that purpose we show that 
there exists a continuous function $\psi$ from $\Ch (\A)$ into $\Ch (\B)$ 
such that $\phi\circ\psi$ and $\psi\circ\phi$ are identity functions on 
$\Ch (\A)$ and $\Ch (\B)$ respectively. 
Although we need some consideration since $T$ needs not be 
injective, a proof of the existence of 
$\psi$ is similar to that of $\phi$. Let $x\in \Ch (\A)$. 
Put
\begin{multline*}
K_x=\{y\in Y:\text{$|Tf(y)|=1$ for every }\\
\text{$f\in A^{-1}$ 
with $|f(x)|=1=\|f\|_{\infty(X)}$}\}.
\end{multline*}
Suppose that 
$f_1, \dots, f_n\in \A^{-1}$ with $|f_j(x)|=1=\|f_j\|_{\infty (X)}$ for 
every $j=1,\dots, n$. We show that $\cap_{j=1}^n|Tf_j|^{-1}(1)\ne
\emptyset$. It will follow that $K_x\ne \emptyset$ by the finite intersection 
property. Put $f=\prod_{j=1}^nf_j$. Since $|f_j(x)|=1=\|f_j\|_{\infty(X)}$ 
we have that 
$|f(x)|=1=\|f\|_{\infty(X)}$. Put $Tf=F$. 
Since $|T1|=1$ on $\Ch (\B)$, we see that
\begin{multline*}
\|F\|_{\infty (Y)}=\|F\|_{\infty (\Ch (\B))}=
\|FT1\|_{\infty (\Ch (B))} \\
=\|FT1\|_{\infty (Y)}=\|f\|_{\infty (X)}=1
\end{multline*}
Thus there exists a $y_0\in \Ch (\B)$ with $|F(y_0)|=1$. 
Since $|f|^{-1}(1)\subset |f_j|^{-1}(1)$, we have 
that 
\[
\Ch (\B)\cap |F|^{-1}(1)\subset |Tf_j|^{-1}(1).
\] 
(Suppose not. 
There exists $y_1\in \Ch (\B)$ with 
\[
|F(y_1)|=1>|Tf_j(y_1)|
\]
since 
$\|Tf_j\|_{\infty(Y)}=1$. Then there exists an $$H\in P_{\B}^0(y_1)$$ with 
$\|Tf_jH\|_{\infty(Y)}<1$. 
Since $|FH(y_1)|=1$, we have $\|FH\|_{\infty(Y)}=1$. 
Choose an $h\in \A^{-1}$ 
with $Th=H$. Then 
\[
\|f_jh\|_{\infty(X)}=\|Tf_jTh\|_{\infty(Y)}=\|Tf_jH\|_{\infty(Y)}<1. 
\]
On the other hand, 
we have $1=\|FH\|_{\infty(Y)}=\|fh\|_{\infty(X)}$. 
Since $|f|^{-1}(1)\subset |f_j|^{-1}(1)$, 
we have that$\|f_jh\|_{\infty(X)}=\|fh\|_{\infty(X)}=1$, which is a contradiction.) 
Thus we have that 
\[\Ch (\B) \cap |F|^{-1}\subset \cap_{j=1}^n|Tf_j|^{-1}(1). 
\]
It follows that $\cap_{j=1}^n|Tf_j|^{-1}(1)\ne \emptyset$ since 
$\Ch (\B) \cap |F|^{-1}(1)\ne \emptyset$. (There exists a $y_2\in Y$ with 
$|F(y_2)|=1$ since $\|F\|=1$. Then $(\overline{F(y_2)}F+1)/2$ is a 
peak function which peaks on $F^{-1}(F(y_2))$. Thus 
we have that 
\[
y_2\in \Ch (B)\cap F^{-1}(F(y_2))\subset 
\Ch (\B) \cap |F|^{-1}(1). )
\]

Next we show that
$K_x$ is a singleton. 
Suppose that $y_1\in K_x$. Then $|Tf(y_1)|=1$ for every $f\in \A^{-1}$ 
with $|f(x)|=1=\|f\|_{\infty(X)}$. For an $f\in \A^{-1}$ with 
\[
|f(x)|=1=\|f\|_{\infty (X)},
\]
put 
\[
\tilde F=(\overline{Tf(y_1)}Tf+1)/2. 
\]
Then $\tilde F$ is a peak function such that $\tilde F(y_1)=1$ 
since $\|Tf\|_{\infty (Y)}=1$. 
Thus $y_1\in \cap \tilde F^{-1}(1)\subset K_x$, where $\cap$ takes for 
all 
$\tilde F=(\overline{Tf(y_1)}Tf+1)/2$ for 
$f\in \A^{-1}$ with $|f(x)|=1=\|f\|_{\infty(X)}$. 
Since $\cap \tilde F^{-1}(1)$ is a
$p$-set, there exists an $y_0\in \left(\cap \tilde F^{-1}(1)\right)\Ch (\B)$. 
We show that $\{y_0\}=K_x$. Suppose that $y_3\in K_x\setminus \{y_0\}$. 
Then there exists an $H\in P_{\B}^0(y_0)$ with $|H(y_3)|<1$. 
Choose an $h\in \A^{-1}$ with $Th=H$. 
Then $\|h\|_{\infty (X)}=\|H\|_{\infty (Y)}=1$. 
Suppose that $|h(x)|=1$. Then 
$|H|=1$ on $K_x$ by the definition of $K_x$, which contradicts to 
$|H(y_3)|<1$. Thus we have that $|h(x)|\ne 1$, so $|h(x)|<1$ since 
$\|h\|_{\infty(X)}=\|H\|_{\infty(Y)}=1$. 
Thus there exists an $f\in P_{\A}^0(x)$  with 
$\|fh\|_{\infty(X)}<1$, so $\|TfH\|_{\infty(Y)}=\|fh\|_{\infty(X)}<1$. 
On the other hand we have that
$|Tf(y_0)|=1$ since $|Tf|=1$ on $K_x$ by the definition of $K_x$. 
Thus we have that $\|TfH\|_{\infty(Y)}=1$, 
which is a contradiction proving that
$K_x=\{y_0\}$. 

Put a function $\psi$ from $\Ch (\A)$ into $\Ch (\B)$ by 
$\psi (x)=y_0$, the unique element of $K_x$. Then by the definition of 
$K_x$, we have that $|Tf(\psi (x))|=|f(x)|$ for every $f\in \A^{-1}$ 
with 
$|f(x)|=1=\|f\|_{\infty(X)}$. 
We show that $|Tf(\psi (x))|=|f(x)|$ holds 
for every $f\in \A^{-1}$. Let $f\in \A^{-1}$. Then there exists an 
$h\in P_{\A}^0(x)$ with $\sigma_{\pi}(fh)=\{f(x)\}$. Thus we see that
\[
|f(x)|=\|fh\|_{\infty(X)}=\|TfTh\|_{\infty(Y)}
\ge |Tf(\psi (x))Th(\psi (x))|.
\]
Since $|Th(\psi (x))|=1$ by the definition of $K_x$, we see that
$|Tf(\psi (x))|\le |f(x)|$. On the other hand, there exists an $H'\in 
P_{\B}^0(\psi (x))$ with 
$$\sigma_{\pi}(TfH')=\{Tf(\psi (x))\}.$$ 
Choose a function $h'\in \A^{-1}$ with $Th'=H'$. Then 
$$\|h'\|_{\infty (X)}=\|H'\|_{\infty (Y)}=1.$$ 
We also see that $|h'(x)|=1$. 
(Suppose not. Then $|h'(x)|<1$. Then 
there exists an $h'' \in P_{\A}^0(x)$ with $\|h'h''\|<1$. 
We also have that 
\[
\|h'h''\|_{\infty(X)}=\|H'Th''\|_{\infty(Y)}
\ge |H'(\psi (x))||Th''(\psi (x))|.
\] 
By the definition of $K_x$, we have that $|Th''(\psi (x))|=1$ 
since $h''(x)=1=\|h''\|_{\infty(X)}$. 
Since $H'\in P_{\B}^0(\psi (x))$, we have that 
$|H'(\psi (x))|=1$, so $\|h'h''\|_{\infty (X)}\ge 1$, 
which is a contradiction. ) 
Thus we have that
\[
|Tf(\psi (x))|=\|TfH'\|_{\infty(Y)}=
\|fh'\|_{\infty(X)}\ge |f(x)h'(x)|=|f(x)|,
\]
so that $|Tf(\psi (x))|=|f(x)|$. 

We show that $\psi$ is continuous. Suppose that $x\in \Ch (\A)$ and 
$\{x_{\alpha}\}$ is a net which converges to $x$. Then for 
every $f\in \A^{=1}$ we have that 
\[
|Tf(\psi (x_{\alpha}))|=|f(x_{\alpha})| \to |f(x)|=|Tf(x_{\alpha})|.
\]
Since $T$ is a surjection, we see that 
$|F(\psi (x_{\alpha}))| \to |F(x)|$ holds for every $F\in \B^{-1}$. 
By the Alexandroff theorem the original topology on $Y$ and the 
weak topology on $Y$ induced by the family $\{|F|: F\in \B^{-1}\}$ 
coincides. So we see that 
$\psi (x_{\alpha}) \to \psi (x)$. Thus we see that $\psi$ is continuous on 
$\Ch (\A)$. 

By the first part of the proof we have that
$|Tf(y)|=|f(\phi(y))|$ holds for every $f\in \A^{-1}$ and $y\in 
\Ch (\B)$, so we see that 
\[
|Tf(y)|=|f(\phi (y))|=|Tf(\psi(\phi(y)))|
\]
hold for 
every $f\in \A^{-1}$. Since $T(\A^{-1})=\B^{-1}$ and $\{|F|:F\in \B^{-1}\}$ 
separates the points of $Y$, we see that$\psi\circ\phi(y)=y$ 
holds for every $y\in \Ch (\B)$. 
Since 
\[
|f(x)|=|Tf(\psi (x))|=|f(\phi (\psi (x)))|
\]
hold for every $x\in \Ch (\A)$, 
in a way similar, we also see that 
$\phi\circ\psi (x)=x$ holds for every $x\in \Ch (\A)$. It follows that
$\phi$ and $\psi$ are bijections and $\phi ^{-1}=\psi$. Since $\phi$ and 
$\psi$ are continuous, we see that $\phi$ is a homeomorphism from 
$\Ch (\B)$ onto $\Ch (\A)$. 
\end{proof}

Note that under the hypotheses of Theorem \ref{0korovkin} the map $T$ need not 
be extended to a linear map from $\A$ into $\B$. 
On the other hand we show that two unital commutative 
C$^*$-algebras are algebraically isomorphic to each other if there 
exists a surjective group homomorphisms  
which preserves the norm (cf. Corollary \ref{cx}). 
The results compares with the following 
example which is presented in \cite[Remark 1.7.8.]{ze}.

\begin{example}\cite{ze}\label{zelazko}
Let $X_1=[0,1] \cup \{2 \}$ and $X_2=[-1,-\frac12 ] \cup [\frac12,1]$. 
Suppose that $T$ is a group isomorphism from $C(X_1)^{-1}$ onto 
$C(X_2)^{-1}$ such that

\begin{equation*}
Tf(y)=
\begin{cases}
f(y+1), &y\in [-1,-\frac12] \\
f(2)f(y), & y\in [\frac12, 1].
\end{cases}
\end{equation*}
Then $T$ is a homeomorphism with respect to the relative 
topologies on $C(X_1)^{-1}$ and $C(X_2)^{-1}$ which are induced by 
the supremum norms on $C(X_1)$ and $C(X_2)$ respectively. 
On the other hand $C(X_1)$ is not algebraically isomorphic to 
$C(X_2)$ since $X_1$ and $X_2$ is not homeomorphic.
\end{example}

In this example $\sup \frac{\|Tf\|_{\infty (X_2)}}
{\|f\|_{\infty (X_1)}}=\infty$ and $\inf
\frac{\|Tf\|_{\infty (X_2)}}{\|f\|_{\infty (X_1)}}=0$. 
In Corollary \ref{cx} we show that if a group homomorphism from 
$C(X)^{-1}$ onto $C(Y)^{-1}$ for compact Hausdorff spaces $X$ and $Y$ 
satisfies that 
$\sup \frac{\|Tf\|_{\infty (X_2)}}
{\|f\|_{\infty (X_1)}}<\infty$ and $\inf
\frac{\|Tf\|_{\infty (X_2)}}{\|f\|_{\infty (X_1)}}>0$, then $C(X)$ is 
algebraically isomorphic to $C(Y)$. 
Note that such a $T$ may not be extended an algebra isomorphism (cf. Example 
\ref{nonisomorphism}).

\begin{cor}\label{normpreserving}
Let $\A$ and $\B$ be uniform algebras on compact Hausdorff spaces 
$X$ and $Y$ respectively. Suppose that $T$ is a group homomorphism 
from $\A^{-1}$ onto $\B^{-1}$ which satisfies that 
$\inf
\frac{\|Tf\|_{\infty (X_2)}}{\|f\|_{\infty (X_1)}}>0$ and 
$\sup \frac{\|Tf\|_{\infty (X_2)}}
{\|f\|_{\infty (X_1)}}<\infty$
Then there is a homeomorphism $\phi$ from $\Ch (\B)$ onto $\Ch (\A)$ 
such that an equality $|Tf(y)|=|f(\phi (y))|$ holds for every 
$f\in \A^{-1}$ and $y\in \Ch (\B)$. 
\end{cor}

\begin{proof}
First we show that $\|Tf\|_{\infty (Y)}=\|f\|_{\infty (X)}$ holds for 
every $f\in \A$. 
Suppose not. If $\|Tf_1\|_{\infty (Y)}>\|f_1\|_{\infty (X)}$ for some 
$f_1\in \A^{-1}$, then we have that $\lim_{n\to\infty}
\frac{\|Tf_1^n\|_{\infty (Y)}}
{\|f_1^n\|_{\infty (X)}}=\infty$ for $T$ is multiplicative. 
If $\|Tf_2\|_{\infty (Y)}<\|f_2\|_{\infty (X)}$ for some 
$f_2\in \A^{-1}$, then we have that $\lim_{n\to\infty}
\frac{\|Tf_2^n\|_{\infty (Y)}}
{\|f_2^n\|_{\infty (X)}}=0$. In any case we have a contradiction. 
Thus we see that $\|Tf\|_{\infty (Y)}=\|f\|_{\infty (X)}$ for every 
$f\in \A^{-1}$.
Since $T$ preserves multiplication we see that
\[
\|TfTg\|_{\infty (Y)}=\|fg\|_{\infty (X)}
\]
holds for every pair $f$ and $g$ in $\A$.
Then by Theorem \ref{0korovkin} we see that the conclusion holds. 
\end{proof}

\begin{cor}\label{cx}
Suppose that $X$ and $Y$ are compact Hausdorff spaces and 
$T:C(X)^{-1}\to C(Y)^{-1}$ is a surjective group homomorphism. 
If $\sup\frac{\|Tf\|_{\infty (Y)}}{\|f\|_{\infty (X)}}<\infty$ and 
$\inf\frac{\|Tf\|_{\infty (Y)}}{\|f\|_{\infty (X)}}>0$, then $C(X)$ is 
isometrically and algebraically isomorphic to $C(Y)$. 
\end{cor}

\begin{proof}
By Proposition \ref{normpreserving} there is a homeomorphism from 
$X$ onto $Y$, so that 
$C(X)$ is isometrically and algebraically isomorphic to $C(Y)$.
\end{proof}

Note that a norm preserving group homomorphism from $C(X)^{-1}$ onto 
$C(Y)^{-1}$ need not be injective. 

\begin{example}\label{new}
Let $X$ be the closed unit interval 
of the real numbers. Put $T:C(X)^{-1}\to C(X)^{-1}$ be defined as 
$T(f)=f^2/|f|$ for each $f\in C(X)^{-1}$. Then $T$ is 
a norm-preserving group homomorphism. 
For any $F\in C(X)^{-1}$, choose $f\in C(X)^{-1}$ with 
$f^2=F|F|$. Such an $f$ exists 
since $C(X)^{-1}$ is closed under 
the square-root-operation (cf. \cite{hm}), 
that is, for every $h\in C(X)^{-1}$ there 
exists $g\in C(X)^{-1}$ with $g^2=h$. 
Then $Tf=F$, so we see that 
 $T$ is a surjection onto $C(X)^{-1}$.
Since $T(f)=T(-f)$, $T$ is not injective. 
\end{example}

Note also that the map $T$ in Corollary \ref{cx} 
need not be extended to an algebra isomorphism 
from $C(X)$ onto $C(Y)$ even if $T$ is injective. 
An example is as follows.

\begin{example}\label{nonisomorphism}
Let $C(X)$ be the usual Banach algebra of complex-valued continuous 
functions on a compact Hausdorff space $X$ and $A$ the 
direct sum $C(X)\oplus C(X)$. 
Note that $A$ is isometrically isomorphic to 
$C(X_1\cup X_2)$, where $X_1$ and $X_2$ are two copies of $X$. 
Let $T$ be a map from $A^{-1}$ into itself defined by 
\[
T(f\oplus g)= \frac{f^2g}{|fg|}\oplus \frac{f^3g^2}{|f^3g|} \quad 
f\oplus g \in A.
\]
Then $T$ is a norm preserving group automorphism on $A^{-1}$ while 
$T$ is not extended to a linear map on $A$.
\end{example}

\begin{proof}
Clearly $T$ is a norm preserving group endomorphism on $A^{-1}$. 
We only show that $T$ is a bijection. 
Let $h_1\oplus h_2$ be an arbitrary function in $A^{-1}$. 
Put $f=(h_1^2|h_2|)/(h_2|h_1|)$ and $g=(h_2^2|h_1|^3)/(h_1^3|h_2|)$. 
Then by a simple calculation that $T(f\oplus g)=h_1\oplus h_2$ and 
this $f\oplus g$ is the only a function with 
$T(f\oplus g)=h_1\oplus h_2$.
\end{proof}

Even if a group isomorphism between the invertible groups of uniform 
algebras preserve the norm, it can be discontinuous.

\begin{example}
Let $C_{\mathbb R}([0,1])$ denote the real Banach space of all real-valued 
continuous functions on the closed unit interval $[0,1]$ and 
$\{u_{\lambda}\}_{\lambda \in \Lambda}$ a basis for 
$C_{\mathbb R}([0,1])$ as a 
real linear space such that $1\in \{u_{\lambda}\}_{\lambda \in \Lambda}$ and 
$\|u_{\lambda}\|_{\infty ([0,1])}=1$ for every $\lambda \in \Lambda$. 
Suppose that $\{u_n\}$ and $\{v_n\}$ are disjoint countable subsets of 
$\{u_{\lambda}\}_{\lambda \in \Lambda}$ and $u_1=1$. 
Without loss of generality we may assume that 
$u_n([0,1])=[0,1]$ and $v_n([0,1])=[0,1]$ for
every positive integer $n$. 
Let $R$ be the linear isomorphism from $C_{\mathbb R}([0,1])$ onto 
itself such that
\begin{equation*}
R(u_{\lambda})=
\begin{cases}
u_{\lambda}, & \text{if $u_{\lambda}\in \{u_{\lambda}\}_{\lambda \in 
\Lambda}\}\setminus \left(\{u_n\}\cup\{v_n\}\right)$} \\
\left(\frac{1}{3n}+\frac{1}{2}\right)v_n, & 
\text{if $u_{\lambda}=v_n$ for some $n$}, \\
nu_n, & 
\text{if $u_{\lambda}=u_n$ for some $n$}.
\end{cases}
\end{equation*}
By a simple calculation we have that $R$ is not a bounded as linear 
transformation on the Banach space $C_{\mathbb R}([0,1])$ and 
$2R-I$ is a linear isomorphism from $C_{\mathbb R}([0,1])$ onto itself, 
where $I$ is the identity operator. Put 
$T:\exp C([0,1]) \to \exp C([0,1])$ defined by 
\[
T(\exp f) = \exp (f-i2R({\mathrm Im}f)), \exp f \in \exp C([0,1]), 
\]
where ${\mathrm Im}f$ denotes the imaginary part of $f$. 
Since $R(u_1)=u_1$ and since $\exp C([0,1])=
(C([0,1]))^{-1}$ by \cite[Corollary III.7.4]{ga}, 
it is easy to see that $T$ is well-defined and 
is a group isomorphism 
form $(C([0,1]))^{-1}$ onto itself such that $\|Tg\|_{\infty ([0,1])} 
=\|g\|_{\infty ([0,1])}$ for every $g\in (C([0,1]))^{-1}$. 
Put $f_n=\frac{i}{\sqrt{n}}u_n$ for every positive integer $n$. 
Then $\|\exp f_n-1\|_{\infty ([0,1])}\to 0$ as $n\to \infty$. On the other 
hand we have that $T(\exp f_n)=\exp ((\frac{1}{\sqrt{n}}-2\sqrt{n})iu_n)$, 
so that $T(\exp f_n) \not\to 1=T1$ 
since $u_n([0,1])=[0,1]$ for every $n$; $T$ is not continuous. Put 
$g_n=-\frac{i}{3\sqrt{n}}v_n$. We see in a way similar to the above that 
$\|\exp g_n-1\|_{\infty ([0,1])}\to 0$ and $T^{-1}(\exp g_n)\not\to 1=
T1$; $T^{-1}$ is not continuous.
\end{example}

\section{Multiplicatively spectrum-preserving maps between invertible 
groups}

The following corollary is a version of a theorem of Luttman and Tonev 
\cite{lt}. Another slight 
generalization of the theorem of Luttman and 
Tonev is given in the  last section (cf. \ref{glt}).

\begin{cor}\label{vlt}
Let $\A$ and $\B$ be uniform algebras on compact Hausdorff spaces 
$X$ and $Y$ respectively and $T$ a map from 
$\A^{-1}$ onto $\B^{-1}$ such that the inclusion
\[
\sigma_{\pi}(TfTg)\subset \sigma_{\pi}(fg)
\]
holds for every pair $f,g\in \A^{-1}$. 
Then $(T1)^2=1$ and there exists a homeomorphism $\phi$ from 
$\Ch (\B)$ onto $\Ch (\A)$ such that the equality 
\[
Tf(y) = T1(y)f(\phi (y)), \quad y\in \Ch (\B)
\]
holds for every $f\in \A^{-1}$. Thus $T/T1$ is extended to an 
isometrical algebra isomorphism from $\A$ onto $\B$. 
In particular, $T$ is extended to an isometrical algebra 
isomorphisms from $\A$ onto $\B$ if $T1=1$. 
\end{cor}

\begin{proof}
First we consider the case where $T1=1$. For every pair $f$ and $g$ in 
$\A^{-1}$ the equality $\|TfTg\|_{\infty (Y)}=\|fg\|_{\infty (X)}$ 
holds since 
$\sigma_{\pi}(TfTg)\subset \sigma_{\pi}(fg)$. 
Then by Theorem \ref{0korovkin} there exists a homeomorphism $\phi$ 
from $\Ch (\B)$ onto $\Ch (\A)$ such that the equality 
$|Tf(y)|=|f(\phi (y))|$ holds for every $y\in \Ch (\A)$ and $f\in \A^{-1}$. 

Let $y\in \Ch (\B)$ and $u\in P_{\A}^0(\phi (y)$. Since we have assumed that 
$T1=1$, we see that
\[
\sigma_{\pi}(Tu)=\sigma_{\pi}(TuT1)\subset \sigma_{\pi}(u\cdot 1)=\{1\},
\]
that is, $Tu \in P_{\B}^0$ and $\sigma_{\pi}(Tu)=\{1\}$. 
On the other hand since $|Tu(y)|=|u(\phi (y))|=1$, we see that 
$Tu(y)=1$ and so $Tu\in P_{\B}^0(y)$. Thus we conclude that 
\[
T(P_{\A}^0(\phi (y)))\subset P_{\B}^0(y). 
\]

Let $f\in \A^{-1}$ and $y\in \Ch (\B)$. 
Then by Lemma \ref{bishop} there exists a $u\in P_{\A}^0(\phi (y))$ 
such that $\sigma_{\pi}(fu)=\{f(\phi (y))\}$.  Then we have that 
$\sigma_{\pi}(TfTu)=\{f(\phi (y))\}$. Thus we see that
\[
|f(\phi (y)|=\|TfTu\|_{\infty (Y)}\ge |Tf(y)Tu(y)|=|Tf(y)|=|f(\phi (y))|
\]
hold since $Tu\in P_{\B}^0(y)$, so $\|TfTu\|_{\infty (Y)}=
|Tf(y)|$. It follows that $Tf(y)=f(\phi (y))$ since 
$\sigma_{\pi}(TfTu)=\{f(\phi (y))\}$. 

We consider the general case and prove that $(T1)^2=1$. First we have that 
$\|T1\|_{\infty (Y)}=1$ since $\sigma_{\pi}(T1T1)\subset 
\sigma_{\pi}(1)=\{1\}$ and $\|(T1)^2\|_{\infty (Y)}=\|T1\|_{\infty (Y)}$. 
Suppose that there exists a $y\in \Ch (\B)$ with $(T1(y))^2\ne 1$. 
Then we see that $|(T1(y)|<1$ since $\sigma_{\pi}((T1)^2)\subset \{1\}$. 
Then by Lemma \ref{bishop} there exists a $U\in P_{\B}^0(y)$ such that 
$\|T1U\|_{\infty (Y)}<1$. Since $T\A^{-1}=\B^{-1}$, 
there exists a $u\in \A^{-1}$ 
with $Tu=U$. Then $\sigma_{\pi}(T1U)=\{1\}$ since 
$\sigma_{\pi}(T1U)\subset \sigma_{\pi}(u1)=\{1\}$ and $\sigma_{\pi}(T1U)
\ne \emptyset$. This contradicts to $\|T1U\|_{\infty (Y)}<1$. 
We conclude that $(T1(y))^2=1$ for every $y\in \Ch (\B)$, and so 
$(T1)^2=1$ for $\Ch (\B)$ is a boundary for $\B$. 
Put $\tilde T=T/T1$. Then $\tilde T$ is a well-defined map from 
$\A^{-1}$ into $\B^{-1}$. By a simple calculation we see that 
$\tilde T(\A^{-1})=\B^{-1}$. We see that 
$\tilde T1=1$ by the definition of $\tilde T$ and that 
\[
\sigma_{\pi}(\tilde Tf\tilde Tg)\subset \sigma_{\pi}(fg)
\]
holds for every pair $f$ and $g$ in $\A^{-1}$ since $(T1)^2=1$, . 
Then by the first part of the 
proof there exists a homeomorphism $\phi$ from $\Ch (\B)$ onto 
$\Ch (\A)$ such that 
\[
\tilde Tf(y)=f(\phi (y)), \quad y\in \Ch (\B)
\]
holds for every $f\in \A^{-1}$. Thus we see that
\[
Tf(y)=T1(y)f(\phi (y)), \quad y\in \Ch (\B)
\]
holds for every $f\in \A^{-1}$. 
Since $\Ch (\B)$ is a boundary for $\B$, the restriction map $R:\B \to 
\B|\Ch (\B)$ is a bijective isometrical algebra isomorphism. 
Put $T_e:\A\to \B|\Ch (\B)$ by 
$T_ef(y)=f(\phi (y)$ for $f\in \A$. Then $T_e$ is well-defined and 
it is easy to see that $R^{-1}\circ T_e$ is an isometrical algebra 
isomorphism from $\A$ onto $\B$ which is an extension of $T$. 
\end{proof}

\begin{cor}\label{gh1}
Let $\A$ and $\B$ be uniform algebras and 
$T$ a group homomorphism from $\A^{-1}$ onto $\B^{-1}$ which satisfies that 
\[
\sigma_{\pi}(Tf)\subset \sigma_{\pi}(f)
\]
holds for every $f\in \A^{-1}$. 
Then $T$ is extended to an isometrical algebra isomorphism from 
$\A$ onto $\B$.
\end{cor}

\begin{proof}
We see that $T1=1$ since $T$ is a surjective group homomorphism. 
We also see that $\sigma_{\pi}(TfTg)\subset \sigma_{\pi}(fg)$ holds 
for every pair $f$ and $g$ in $\A^{-1}$ since $TfTg=Tfg$. 
The conclusion follows from Corollary \ref{vlt}.

\end{proof}

\begin{cor}\label{commutative}
Let $A$ be a unital semisimple commutative Banach algebra and 
$B$ a unital commutative Banach algebra. Suppose that 
$T$ is a surjective group homomorphism from $A^{-1}$ onto 
$B^{-1}$ such that 
\[
\sigma (Tf)=\sigma (f)
\]
holds for every $f\in A^{-1}$. 
Then $T$ is extended to an algebra isomorphism from $A$ onto $B$. 
In particular, $B$ is semisimple.
\end{cor}

\begin{proof}
Since 
\[
\sigma(\frac{Tf}{Tg})=\sigma(T(\frac{f}{g}))=\sigma (\frac{f}{g})
\]
hold for every pair $f$ and $g$ in $A^{-1}$, we 
see that $T$ is injective. 

First we consider the case where $B$ is semisimple. 
Recall that  
we denote the uniform closure of the Gelfand transform 
of $A$ in $C(M_A)$ (resp. $B$ in $C(M_B)$) by $\mcl A$ 
(resp. $\mcl B$). We may consider that 
$A \subset \mcl A$ (resp. $B \subset \mcl B$) 
since $A$ (resp. $B$) is semisimple. Note that the maximal 
ideal space $M_{\mcl A}$ (resp. $M_{\mcl B}$) 
is homeomorphic to $M_A$ (resp. $M_B$). In fact, 
the correspondence 
between complex homomorphisms on $\mcl A$ (resp. $\mcl B$) 
and its restrictions 
on $A$ (resp. $B$) gives a homeomorphism between $M_{\mcl A}$ and $M_A$ 
(resp. $M_{\mcl B}$ and $M_B$). 
Thus the spectrum of $f\in \mcl A$ (resp. 
$f\in \mcl B$) is coincide with $f(M_A)$ (resp. $f(M_B)$); 
we denote the spectrum with respect to $\mcl A$ (resp. $\mcl B$) 
also by $\sigma (f)$.

We extend $T$ to a group homomorphism from $(\mcl A)^{-1}$ 
onto $(\mcl B)^{-1}$ such that $\sigma (Tf)=\sigma (f)$ holds for 
every $f\in (\mcl A)^{-1}$. Suppose that $f \in (\mcl A)^{-1}$. 
Then there is a sequence $\{f_n\}$ in $A$ such that 
$\|f_n-f\|_{\infty (M_A)}\to 0$ as $n\to \infty$. 
We may assume that $\{f_n\}\subset A^{-1}$. 
Since $|f|>0$ on 
$M_A$, we may assume without loss of generality that 
there is a positive number $M$ such that $\frac{1}{M}<|f_n|<M$ on $M_A$. 
So $|\frac{f_n}{f_m}-1|\le M|f_m-f_n|$ on $M_A$. Since $T$ is a 
group homomorphism with the assumption concerning the spectrum we see that 
\[
\sigma(\frac{Tf_n}{Tf_m})=\sigma(T(\frac{f_n}{f_m}))=\sigma(\frac{f_n}{f_m}),
\]
so that $\|\frac{Tf_n}{Tf_m}-1\|_{\infty (M_B)}=
\|\frac{f_n}{f_m}-1\|_{\infty (M_A)}$.
It follows that
\begin{multline*}
\|Tf_n-Tf_m\|_{\infty (M_B)}\le \|\frac{Tf_n}{Tf_m}-1\|_{\infty (M_B)}
\|Tf_m\|_{\infty (M_B)}\\
\le M^2\|f_n-f_m\|_{\infty (M_A)}
\end{multline*}
since $\|Tf_m\|_{\infty (M_B)}=\|f_m\|_{\infty (M_A)}$. 
Therefore $\{Tf_n\}$ is a Cauchy sequence 
with respect to the supremum norm and the uniform limit $F$ is 
in $\mcl B$. Since $\sigma (Tf_n)=\sigma (f_n)$ and $\frac{1}{M}<
|f_n|<M$  holds on $M_A$ for every $n$, we see that 
$\frac{1}{M}<|Tf_n|<M$ holds on $M_B$ for every $n$, so $|F|\ge \frac{1}{M}$ 
on $M_B$. Thus we see that $F\in (\mcl B)^{-1}$ 
since $M_{\mcl B}=M_B$. By a routine calculation 
the function $F$ is independent of the choice of the sequence $\{f_n\}$. 
Put $\widetilde{T}f=F$. Then $\widetilde{T}$ is a function from 
$(\mcl A)^{-1}$ into $(\mcl B)^{-1}$ and $\widetilde{T}=T$ on $A^{-1}$. 

We show that $\widetilde{T}$ is a surjection. Let $F\in (\mcl B)^{-1}$. 
Then there is a sequence $\{F_n\}$ in $B^{-1}$ which uniformly converges to 
$F$. In the way similar to the above, we see that $\{T^{-1}F_n\}$ is 
a Cauchy sequence in $A^{-1}$ and converges uniformly to a function $f\in 
(\mcl A)^{-1}$. Then by the definition of $\widetilde{T}$ we see that 
$\widetilde{T}f=F$.

We show that $\sigma (\widetilde{T}f)=\sigma (f)$ holds for every 
$f\in (\mcl A)^{-1}$. Suppose that $f \in (\mcl A)^{-1}$. Then
there exists a sequence $\{f_n\}$ in $A^{-1}$ such that 
$\|f_n-f\|_{\infty (M_A)} \to 0$ as $n\to \infty$. Then by the definition of 
$\widetilde{T}$ we see that $\|{T}f_n-\widetilde{T}f\|_
{\infty (M_B)}\to 0$ as $n\to \infty$. Suppose that $\lambda \in \sigma (f)$. 
Then there is an $x \in M_A$ with $\lambda= f(x)$. Put 
$\lambda_n=f_n(x)$, so $\lambda_n\to \lambda$ as $n\to \infty$. On the
other hand, for each positive integer $n$, 
there is $y_n\in M_B$ such that $\lambda_n 
=Tf_n(y_n)$ since $\sigma (f_n)=\sigma (Tf_n)$ for every $n$. Thus we have
\[
|\lambda-\widetilde{T}f(y_n)|\le |\lambda-\lambda_n|+ 
\|Tf_n-\widetilde{T}f\|_{\infty (M_B)} \to 0
\]
as $n\to \infty$. Thus we have that $\lambda \in \sigma (\widetilde{T}f)$, 
so that $\sigma (f) \subset \sigma (\widetilde{T}f)$. The reverse inclusion 
is proven in the same way; we see that 
$\sigma (f) = \sigma (\widetilde{T}f)$. 

We show that 
$\widetilde{T}$ is a group homomorphism. Let $f, g\in (\mcl A)^{-1}$. 
Then there are sequences $\{f_n\}$ and $\{g_n\}$ in $A^{-1}$ such that 
$\|f_n-f\|_{\infty (M_A)}\to 0$ and 
$\|g_n-g\|_{\infty (M_B)} \to 0$ as $n\to \infty$. So 
$\|f_ng_n-fg\|_{\infty (M_A)} \to 0$ as $n\to \infty$. 
Then we see that 
\[
\|Tf_n-\widetilde{T}f\|_{\infty (M_B)}\to 0, \quad
\|Tg_n-\widetilde{T}g\|_{\infty (M_B)}\to 0
\]
and so 
\[
\|Tf_nTg_n-\widetilde{T}f\widetilde{T}g\|_{\infty (M_B)} \to 0
\] 
as $n\to \infty$. We also have that 
\[
\|T(f_ng_n)-\widetilde{T}(fg)\|_{\infty (M_B)}\to 0
\]
as 
$n\to \infty$ since $\|f_ng_n-fg\|_{\infty (M_A)} \to 0$ as $n\to \infty$.
It follows that $\widetilde{T}f\widetilde{T}g=\widetilde{T}(fg)$ 
since $T(f_ng_n)=Tf_nTg_n$.

Since $\sigma(\widetilde{T}f)=\sigma (f)$ holds for every $f\in 
(\mcl A)^{-1}$, 
$\sigma_{\pi}(\widetilde{T}f)=\sigma_{\pi} (f)$ holds for every $f\in 
(\mcl A))^{-1}$. Thus applying Corollary \ref{gh1} we see that 
$\widetilde{T}$ is extended to an algebra isomorphism from 
$\mcl A$ onto $\mcl B$. By a simple calculation we see that 
the restriction of the extended isomorphism to $A$ is an algebra 
isomorphism from $A$ onto $B$.

Finally we consider the general case. 
Let $\Gamma_B$ denote the Gelfand transform of $B$. Then by a simple 
calculation we see that $\Gamma_B\circ T$ is a group homomorphism 
from $A^{-1}$ onto $(\Gamma_B(B))^{-1}$, since $\Gamma_B(B^{-1})
=(\Gamma_B(B))^{-1}$. Then by the first part of the proof, we see 
that $\Gamma_B\circ T$ is extended to an algebra isomorphism 
from $A$ onto $\Gamma_B(B)$. Then we see that $\Gamma_B$ is injection 
from $B^{-1}$ onto $(\Gamma_B(B))^{-1}$. (Suppose that 
 $g_1,g_2\in B^{-1}$ with $\Gamma_B(g_1)=\Gamma_B(g_2)$. 
Since $T$ is surjection from $A^{-1}$ onto $B^{-1}$, there are 
$f_1,f_2\in A^{-1}$ with $Tf_1=g_1$ and $Tf_2=g_2$. Then 
we have 
\[
\Gamma_B\circ T(f_1)=\Gamma_B(g_1)=\Gamma_B(g_2)=
\Gamma_B\circ T(f_2).
\]
Since $\Gamma_B\circ T$ is an injection we see that $f_1=f_2$, thus 
$g_1=g_2$, that is, $\Gamma_B$ is injective on $B^{-1}$. It follows by a 
simple calculation that $\Gamma_B$ is an injection on $B$. 
Thus $B$ is semisimple, and $T$ is extended to an algebra isomorphism from 
$A$ onto $B$ applying the first part of the proof.
\end{proof}

\section{Non-symmetric multiplicatively norm-preserving maps 
between invertible groups}

Under the hypotheses in Theorem \ref{0korovkin}, the map $T$ 
appearing in Theorem \ref{0korovkin} need not be 
linear nor multiplicative. 
For example, let $\A$ be a uniform algebra on $X$ and put a map 
$\varepsilon$ from $\A$  into $\{-1,1\}$. Then the map $T:\A\to \A$ defined by 
$Tf=\varepsilon (f)f$, $f\in \A$ satisfies that the equality 
$\|TfTg\|_{\infty (X)}=\|fg\|_{\infty (X)}$ holds for every 
pair $f$ and $g$ in $\A$ and $T$ can be surjective which is not linear nor 
multiplicative according to the choice of $\varepsilon$. 

In this section we consider non-symmetric multiplicatively norm-preserving 
maps between invertible groups. Let $A$ and $B$ be unital commutative Banach 
algebras and $T$ a map from $A^{-1}$ into $B^{-1}$. We say that $T$ is 
non-symmetric multiplicatively (spectral) 
norm-preserving if there exists a nonzero 
complex number $\alpha$ such that 
\[
\|\widehat{Tf}\widehat{Tg}-\alpha \|_{\infty (M_B)}=
\|\hat f \hat g-\alpha \|_{\infty (M_A)}
\]
holds for every pair $f$ and $g$ in $A$, where $M_A$ (resp. $M_B$) 
denotes the maximal ideal space of $A$ (resp. $B$). 
Multiplicatively norm-preserving map corresponds to the case where 
$\alpha =0$, Although multiplicatively norm-preserving maps need not be 
extended to linear nor multiplicative maps, we show that non-symmetric ones 
are extended to real-linear and multiplicative maps if $T1=1$ and $T$ is 
surjective. 

In the following, from Lemma \ref{0} to Lemma \ref{8}, 
$\A$ and $\B$ are uniform algebras on compact 
Hausdorff spaces 
$X$ and $Y$ respectively and $T$ is a map from $\A^{-1}$ onto $\B^{-1}$ 
such that
\[
\|TfTg-1\|_{\infty (Y)}=\|fg-1\|_{\infty (X)}
\]
holds for every $f, g \in \A^{-1}$.

\begin{lemma}\label{0}
$T$ is an injection. The equality $\|TfTg\|_{\infty (Y)}=\|fg\|_{\infty (X)}$ 
holds for every 
$f,g\in \A^{-1}$. 
\end{lemma}
\begin{proof}
First we show that $T$ is an injection. 
Suppose that $Tf=Tg$. 
Then we have that
\begin{multline*}
0=\|gg^{-1}\|_{\infty (X)}=\|TgTg^{-1}-1\|_{\infty (Y)}
\\
=
\|TfTg^{-1}-1\|_{\infty (Y)}
=\|fg^{-1}-1\|_{\infty (X)}.
\end{multline*}
Thus we see that $fg^{-1}=1$, so $f=g$.

Next we show that $\|TfTg\|_{\infty(Y)}=\|fg\|_{\infty(X)}$ 
holds for every $f,g \in \A^{-1}$. 
Since 
\[
\|T1T1-1\|_{\infty(Y)}=\|1^2-1\|_{\infty(X)}=0,
\]
we see that $(T1)^2=1$.
Put $\tilde T=\frac{T}{T1}$. Then by a simple calculation 
we see that 
$\tilde T$ is a map from $\A^{-1}$ 
onto $\B^{-1}$ and the equality 
\[
\|\tilde Tf\tilde Tg-1\|_{\infty(y)}=\|fg-1\|_{\infty(X)}
\]
holds
for every $f,g \in \A^{-1}$ since $(T1)^2=1$. We show that the equality
\[
\|\tilde Tf\tilde Tg\|_{\infty(Y)}=\|fg\|_{\infty(X)}
\] 
holds for every $f,g \in \A^{-1}$. It will follow 
that $\|TfTg\|_{\infty(Y)}=\|fg\|_{\infty(X)}$ 
for every $f,g \in \A^{-1}$ since $(T1)^2=1$. 
Since $\tilde T1=1$, we have for every positive integer $n$ that
\[
\|\tilde Tn-1\|_{\infty(Y)}=\|\tilde Tn \,1-1\|_{\infty(Y)}
=\|n\cdot 1-1\|_{\infty(X)}=n-1.
\]
So we have that 
\[
n-2\le \|\tilde Tn\|_{\infty(Y)}\le n.
\]
Let $f\in \A^{-1}$. 
Then we have that $(\tilde Tf)^{-1}=\tilde T(f^{-1})$ since 
\[
\|\tilde Tf\tilde T(f^{-1})-1\|_{\infty(Y)}=\|ff^{-1}-1\|_{\infty(X)}=0.
\]
Put $K_n=\tilde T(nf)\tilde T(f^{-1})$.
We see that
\begin{multline*}
\|K_n-1\|_{\infty(Y)}=\|\tilde T(nf)\tilde T(f^{-1})-1\|_{\infty(Y)} \\
=
\|nff^{-1}-1\|_{\infty(X)}=n-1,
\end{multline*}
so we have that $\|K_n\|_{\infty(Y)}\le n$. 
For every $g\in \A^{-1}$, we have that 
\begin{multline*}
\left|n\|fg\|_{\infty(X)}-1\right|\le \|nfg-1\|_{\infty(X)}\\
=\|\tilde T(nf)\tilde Tg-1\|_{\infty(Y)}
\le \|K_n\|_{\infty(Y)}\|\tilde Tf\tilde Tg\|_{\infty(Y)}+1.
\end{multline*}
It follows that
\[
\|fg\|_{\infty(X)}-\frac1n\le \frac{\|K_n\|_{\infty(Y)}}{n}
\|\tilde Tf\tilde Tg\|_{\infty(Y)}+\frac1n
\le \|\tilde Tf\tilde Tg\|_{\infty(Y)}+\frac1n,
\]
and letting $n\to \infty$, we have that the inequality 
$\|fg\|_{\infty(X)}\le \|\tilde Tf\tilde Tg\|_{\infty(Y)}$ 
holds for every $f,g \in \A^{-1}$. 
$\tilde T$ is injective since $T$ is. Applying the similar argument to 
$(\tilde T)^{-1}$ instead of $\tilde T$, we have that 
\[
\|FG\|_{\infty(Y)}\le \|(\tilde T)^{-1}F(\tilde T)^{-1}G\|_{\infty(X)}
\]
holds for every $F,G 
\in \B^{-1}$. It follows that the 
equality $\|fg\|_{\infty(X)}=\|\tilde Tf\tilde Tg\|_{\infty(Y)}$ 
holds for every 
$f,g \in \A^{-1}$. Since $(T1)^2=1$ we conclude that
the equality $\|TfTg\|_{\infty(Y)}=\|fg\|_{\infty(X)}$ 
holds for every $f,g\in \A^{-1}$. 
\end{proof}

By Theorem \ref{0korovkin} we see that there exists a homeomorphism 
$\phi$ from $\Ch (\B)$ onto $\Ch (\A)$ such that 
\[
|Tf(y)|=|f(\phi (y))|, \quad y\in \Ch (\B)
\]
holds for every $f\in \A$. 
In the following up to Lemma \ref{8} $\phi$ denotes this 
homeomorphism.
Moreover we see that the following.

\begin{lemma}\label{5.15}
$|T\lambda|=|\lambda|$ on $\Ch (\B)$ and $|T^{-1}\lambda |=
|\lambda|$ on $\Ch (\A)$ for every complex number $\lambda$. 
\end{lemma}

\begin{lemma}\label{1}
Suppose that $T1=1$. For every complex number $\lambda$ with $|\lambda|=1$ 
and $\lambda \ne 1, -1$, we have that 
$$(T^{-1}(\lambda))(\Ch (\A))\subset 
\Lambda_1\cup \Lambda_2,
$$ where $\Lambda_1$ (resp. $\Lambda_2$) 
is the closed arc on the closed unit circle with the end points 
$\lambda$ and $-\bar{\lambda}$ (resp. $\bar{\lambda}$ and 
$-\lambda$) which does not contain the real number.
\end{lemma}
\begin{proof}
Note that $T^{-1}$ is well-defined since $T$ is an injection by Lemma 
\ref{0}. 
Then we have
\begin{equation}\label{ykg}
\|T^{-1}\lambda -1\|_{\infty(X)}=\|\lambda -1\|_{\infty(Y)}=|\lambda -1|.
\end{equation}
Since 
\[
\|T(-1)T(-1)-1\|_{\infty(Y)}=\|(-1)^2-1\|_{\infty(X)}=0,
\] 
we have that $(T(-1))^2=1$. 
Since $T$ is injective by Lemma \ref{0}, $T(-1)\ne 1$, so 
there exists $y\in \Ch (\B)$ such that $(T(-1))(y)=-1$.
We have that
\begin{equation}\label{ykg2}
\|T^{-1}\lambda +1\|_{\infty(X)}=\|-T^{-1}\lambda -1\|_{\infty(X)}=
\|T(-1)\lambda-1\|_{\infty (Y)}.
\end{equation}
Suppose that $\mathrm{Re}\lambda \ge 0$. We see that
\[
\|T(-1)\lambda-1\|_{\infty(Y)}=|\lambda+1|
\] 
since $(T(-1))^2=1$ and $T(-1)$ takes the value $-1$. 
Thus we have by the equation (\ref{ykg2}) 
that $\|T^{-1}\lambda +1\|_{\infty(X)}=|\lambda +1|$ if 
$\mathrm{Re}\lambda \ge 0$.
Recall that  
$\|T^{-1}\lambda -1\|_{\infty(X)}=|\lambda -1|$ by the equation (\ref{ykg}) 
and
 $|T^{-1}\lambda|=|\lambda|$ on $\Ch (\A)$ by Lemma \ref{5.15}. 
It follows that $(T^{-1}\lambda)(\Ch (\A))\subset \{\lambda, \bar{\lambda}\}$
if $\mathrm{Re}\lambda \ge 0$. Suppose that 
$\mathrm{Re}\lambda < 0$. Then 
$\|T(-1)\lambda -1\|_{\infty (Y)} \le |\lambda-1|$ 
since $\left(T(-1)\right)^2=1$. 
Thus by the equations (\ref{ykg}) and (\ref{ykg2}) we see that
\[
(T^{-1}(\lambda))(\Ch (\A))\subset \Lambda_1\cup \Lambda_2.
\] 
In any case we have the conclusion. 
\end{proof}

\begin{lemma}\label{-2}
Suppose that $T1=1$. Then $T(P_\A^0(\phi (y)))=P_\B^0(y)$ holds for 
every $y\in \Ch (\B)$.
\end{lemma}
\begin{proof}
Suppose that $u\in P_\A^0(\phi (y))$. 
Then by Lemma \ref{0} we see that 
\[
\|Tu\|_{\infty (Y)}=\|u\|_{\infty (X)}=1
\]
since $T1=1$. First we show that 
$\sigma_{\pi}(Tu)=\{1\}$. Suppose that $\alpha \in 
\sigma_{\pi}(Tu)\setminus \{1\}$. We have that
\[
\|Tu-1\|_{\infty(Y)}=\|u-1\|_{\infty(X)}<2
\]
since $T1=1$ and $\sigma_{\pi}(u)=\{1\}$, so
$\alpha \ne -1$. Thus 
\[
\|T^{-1}(-\bar{\alpha})u-1\|_{\infty(X)}=
\|T^{-1}(-\bar{\alpha})u-1\|_{\infty (\Ch (\A))}
<2
\]
by Lemma \ref{1}. On the other hand, $\|-\bar{\alpha}Tu-1\|_{\infty(Y)}=2$ 
since $\alpha \in \sigma_{\pi}(Tu)$, which is a contradiction since 
\[
\|-\bar{\alpha}Tu-1\|_{\infty(Y)}=\|T^{-1}(-\bar{\alpha})u-1\|_{\infty(X)}.
\]
We have that
$\sigma_{\pi}(Tu)=\{1\}$ since $\sigma_{\pi}(Tu)$ is not empty. 
By Theorem \ref{0korovkin} and Lemma \ref{0} we have that 
\[
|Tu(y)|=|u(\phi (y))|=1,
\] 
so that $Tu(y)=1$ since $\sigma_{\pi}(Tu)=\{1\}$. We have proved that 
$Tu\in P_\B^0(y)$ for every $u\in P_\A^0(\phi (y))$. 
Since $T$ is injection by 
Lemma \ref{0}, we have in a way similar to the 
above that 
$T^{-1}(P_\B^0(y))\subset P_\A^0(\phi(y))$. Then the conclusion holds. 
\end{proof}

\begin{lemma}\label{30}
Suppose that $T1=1$. Then $T(-1)=-1$. 
\end{lemma}
\begin{proof}
Since 
\[
\|T(-1)T(-1)-1\|_{\infty(Y)}=\|(-1)^2-1\|_{\infty(X)}=0,
\]
we see that $(T(-1))^2=1$. 
Suppose that there is a $y\in \Ch (\B)$ such that 
$(T(-1))(y)=1$. Then there exists a $U \in P_\B^0(y)$ with 
$\sigma_{\pi}(T(-1)U)=\{1\}$ by Lemma \ref{bishop}. Thus 
\[
\|-T^{-1}U-1\|_{\infty(X)}=
\|T(-1)U-1\|_{\infty(Y)}<2.
\] 
On the other hand $T^{-1}U\in P_\A^0(\phi (y))$ by Lemma \ref{-2}, so 
\[
\|T(-1)U-1\|_{\infty(Y)}=\|-T^{-1}U-1\|_{\infty(X)}=2,
\]
which is a contradiction.
\end{proof}

\begin{lemma}\label{4}
Suppose that $T1=1$. Then $(T\lambda )(\Ch (\B))\subset 
\{\lambda, \bar{\lambda}\}$ for every complex number $\lambda$ with 
the unit absolute value.
\end{lemma}
\begin{proof}
By Lemma \ref{5.15} we see that
$|T\lambda|=|\lambda|$ holds on $\Ch (\B)$. Since $T1=1$, 
we have 
$\|T\lambda -1\|_{\infty(Y)}=|\lambda -1|$. 
Since $T(-1)=-1$ by Lemma \ref{30} we also see that
\begin{multline*}
\|T\lambda +1\|_{\infty(Y)}=\|-T\lambda-1\|_{\infty(Y)}=
\|T(-1)T\lambda-1\|_{\infty(Y)} \\
=|-\lambda-1|=|\lambda+1|.
\end{multline*}
It follows by a simple calculation that the conclusion holds.
\end{proof}

\begin{lemma}\label{5}
Suppose that $T1=1$. Then $T(-i)=-Ti$.
\end{lemma}
\begin{proof}
Since $\|T(-i)Ti-1\|_{\infty(Y)}=\|-i\cdot i-1\|_{\infty(X)}=0$, we have that
$T(-i)Ti=1$. By Lemma \ref{4} we see 
that $(Ti)^2(\Ch (\A))=\{-1\}$, so $(Ti)^2=-1$. Thus we see that 
the conclusion holds.
\end{proof}

\begin{definition}\label{K}
Suppose that $T1=1$. Put 
\[
K=\{y\in \Ch (\B):Ti(y)=i\}.
\]
\end{definition}

\begin{lemma}\label{6}
Suppose that $T1=1$. Then $K$ is a clopen subset of $\Ch (\B)$ 
and $Ti=-i$ on $\Ch (\B)\setminus K$.
\end{lemma}
\begin{proof}
Since $Ti$ is continuous on $\Ch (\B)$ and 
$Ti(\Ch (\B))\subset \{i,-i\}$ by Lemma \ref{4}, 
$K$ is a clopen subset of $\Ch (\B)$ and $T(i)=-i$ on $\Ch (\B)\setminus K$.
\end{proof}

\begin{lemma}\label{7}
Suppose that $T1=1$. For every complex number $\alpha$ with the 
absolute value $1$, $T\alpha (y)=\alpha$ if $y\in K$ and 
$T\alpha (y)=\bar{\alpha}$ if $y\in \Ch (\B)\setminus K$. 
Thus $T(\alpha \beta)=T\alpha T\beta$ holds for every pair of 
complex numbers $\alpha$ and $\beta$ with unit absolute values.
\end{lemma}

\begin{proof}
If $\alpha =-1$, then $T(\alpha)=\alpha$ by Lemma \ref{30}. 
We consider the case where $\alpha$ is a imaginary number. 
By Lemma \ref{5} we have that 
\begin{multline*}
\|TiT\alpha +1\|_{\infty(Y)}=\|-TiT\alpha -1\|_{\infty(Y)}=
\|T(-i)T\alpha -1\|_{\infty(Y)} \\
=\|-i\alpha-1\|_{\infty(X)}=\|i\alpha+1\|_{\infty(X)}.
\end{multline*}
We also see that 
$\|TiT\alpha -1\|_{\infty(Y)}=\|i\alpha-1\|_{\infty(X)}$. 
By  
Lemma \ref{5.15} we see that $|TiT\alpha |=1$ on $\Ch (\B)$. 
It follows by a simple calculation that 
\[
(TiT\alpha)(\Ch (\B))\subset \{i\alpha, \overline{i\alpha}\}.
\] 
Suppose that $y\in K$. Then $Ti(y)=i$. By Lemma \ref{4}, 
$T\alpha (y) = \alpha$ or $\bar{\alpha}$.
Suppose that 
$T\alpha (y)=\bar{\alpha}$. Then $(TiT\alpha)(y)=i\bar{\alpha}$, which 
contradicts to $(TiT\alpha)(\Ch (\B))\subset \{i\alpha, \overline{i\alpha}\}$. 
Thus we see that $T\alpha (y)=\alpha$ if $y\in K$. In a way similar, we see 
that $T\alpha (y)=\bar{\alpha}$ if $y\in \Ch (\B)\setminus K$. 
Thus $T(\alpha \beta)=T(\alpha )T(\beta)$ holds on $\Ch (\B)$ and so 
we have that $T(\alpha \beta)=T(\alpha )T(\beta)$. 
\end{proof}

\begin{lemma}\label{8}
Suppose that $T1=1$. Then 
$T(\alpha P_\A^0(\phi (y)))=(T\alpha)P_\B^0(y)$ holds for every $y\in 
\Ch (\B)$ and every complex number $\alpha$ with the unit absolute value.
\end{lemma}
\begin{proof}
First we show that
$T(\alpha P_\A^0)\subset (T\alpha)P_\B^0$. 
Suppose that $u\in P_\A^0$. Since 
\[
2=\|-\bar{\alpha}\alpha u -1\|_{\infty(X)}
=\|T(-\bar{\alpha})T(\alpha u)-1\|_{\infty(Y)},
\]
we have that 
$-1\in \sigma_{\pi}(T(-\bar{\alpha})T(\alpha u))$.
Suppose that 
\[
\beta
\in \sigma_{\pi}(T(-\bar{\alpha})T(\alpha u))\setminus \{-1\}.
\]
Note that $|\beta|=1$ since 
\[
\|T(-\bar\alpha)T(\alpha u)\|_{\infty (Y)}=1.
\]
Since 
$\sigma_{\pi}(F)\subset F(\Ch (\B))$ holds for every $F\in \B$, 
there exists a $y\in \Ch (\B)$ with $(T(-\bar{\alpha})T(\alpha u))(y)
=\beta$. If $y\in K$, then 
\[
(T((-\bar{\beta})(-\bar{\alpha}))T(\alpha u))(y)=
(T(-\bar{\beta})T(-\bar{\alpha})T(\alpha u))(y)=-1
\]
since $T(-\bar{\beta})=-\bar{\beta}$ on $K$ by Lemma \ref{7}, 
so that 
\[
2=\|T((-\bar{\beta})(-\bar{\alpha}))T(\alpha u)-1\|_{\infty(Y)}=
\|\bar{\beta}u-1\|_{\infty(X)}.
\]
Since $u\in P_\A^0$ we see that $\bar{\beta}=-1$, so $\beta=-1$, 
which is a contradiction. If $y\in \Ch (\B)\setminus K$, then 
we have, in a way similar to the above, a contradiction. 
It follows that 
\[
\sigma_{\pi}(T(-\bar{\alpha})T(\alpha u))=\{-1\}.
\] 
Thus we see that 
$T(-\bar{\alpha})T(\alpha P_\A^0)\subset -P_\B^0$. 
Since 
\[
T(-\alpha)T(-\bar{\alpha})=T1=1
\]
and 
\[
-T(-\alpha)=T(-1)T(-\alpha)=T(\alpha)
\]
by Lemma \ref{30} and \ref{7}, 
we see that
\[
T(\alpha P_\A^0)\subset T(\alpha)P_\B^0.
\]

Next we show that 
\[
T(\alpha P_\A^0(\phi (y))\subset T(\alpha)P_\B^0(y).
\]
Suppose that $u \in P_\A^0(\phi (y))$. Then by the above we have 
that $T(\alpha u)\in T(\alpha )P_\B^0$. Since $T(\bar{\alpha})
T(\alpha)=T1=1$ by Lemma \ref{7}, we see that
$T(\bar{\alpha})T(\alpha u)\in P_\B^0$. On the other hand, by Theorem 
\ref{0korovkin} and Lemma \ref{0}, 
\[
|(T(\bar{\alpha})T(\alpha u))(y)|=|\bar{\alpha}(\phi (y))|
|\alpha u(\phi (y))|=1,
\]
so that $(T(\bar{\alpha})T(\alpha u))(y)=1$
since $T(\bar{\alpha})T(\alpha u) \in P_\B^0$. It follows that 
$T(\bar{\alpha})T(\alpha u)\in P_\B^0(y)$ and so 
$T(\alpha u) \in (T\alpha)P_\B^0(y)$. We see that
\[
T(\alpha P_\A^0(\phi (y))\subset T(\alpha)P_\B^0(y).
\]

We show that 
\[
T(\alpha P_\A^0(\phi (y))\supset T(\alpha)P_\B^0(y).
\]
Suppose that $U\in P_\B^0(y)$. Then
\begin{multline*}
\|-\bar{\alpha}T^{-1}((T\alpha) U)-1\|_{\infty(X)}=
\|T(-\bar{\alpha})T(\alpha )U-1\|_{\infty(Y)}\\
=
\|-U-1\|_{\infty(Y)}=2,
\end{multline*}
so 
$-1\in \sigma_{\pi}(-\bar{\alpha}T^{-1}((T\alpha)U)$. 
Suppose that $\beta \in \sigma_{\pi}(-\bar{\alpha}T^{-1}((T\alpha)U)$.
Then we have that $|\beta|=1$. We also have 
\begin{multline*}
2=\|(-\bar{\beta})(-\bar{\alpha}T^{-1}((T\alpha)U))-1\|_{\infty(X)}\\
=
\|T((-\bar{\beta})(-\bar{\alpha}))(T\alpha)U-1\|_{\infty(Y)} 
=
\|(T((-\bar{\beta})(-\bar{\alpha})\alpha))U-1\|_{\infty(Y)}\\
=
\|(T(\bar{\beta}))U-1\|_{\infty(Y)}.
\end{multline*}
Since $|T(\bar{\beta})|=1$ on $\Ch (\B)$, we see that
$-1\in (T(\bar{\beta}))(\Ch (\B))$, so by Lemma \ref{4} 
we have that 
$\beta=-1$. We see that $\sigma_{\pi}(-\bar{\alpha}T^{-1}((T\alpha)U))
=\{-1\}$, so $-\bar{\alpha}T^{-1}((T\alpha)U)\in -P_\A^0$. 
On the other hand, 
\begin{multline*}
|(-\bar{\alpha}T^{-1}((T\alpha)U))(\phi (y))|=
|T^{-1}((T\alpha)U)(\phi (y))| \\
=|(T(\alpha))U(y)|=|U(y)|=1,
\end{multline*}
so 
\[
(-\bar{\alpha}T^{-1}((T\alpha)U))(\phi (y))=-1
\]
since 
$-\bar{\alpha}T^{-1}((T\alpha)U)\in -P_\A^0$.
Thus 
\[
-\bar{\alpha}T^{-1}((T\alpha)U)\in -P_\A^0(\phi (y)),
\]
so 
\[
T^{-1}((T\alpha)U)\in \alpha P_\A^0(\phi (y)).
\] 
Thus $(T\alpha)(U)\in T(\alpha P_\A^0(\phi (y)))$. 
Since $U\in P_\B^0(y)$ is arbitrary, we see that
$(T\alpha)(P_\B^0(y))\subset T(\alpha P_\A^0(\phi (y)))$.
It follows that 
\[
(T\alpha)(P_\B^0(y))= T(\alpha P_\A^0(\phi (y))).
\]
\end{proof}

\begin{theorem}\label{main}
Let $\A$ and $\B$ be uniform algebras on compact Hausdorff spaces 
$X$ and $Y$ respectively. Suppose that 
$T$ is a map from $\A^{-1}$ onto $\B^{-1}$. Suppose that 
the equality 
\[
\|TfTg-1\|_{\infty(Y)}=\|fg-1\|_{\infty(X)}
\]
holds for every $f,g\in \A^{-1}$. 
Then $(T1)^2=1$ and the map $T$ is extended to a map $T_E$ from 
$\A$ onto $\B$, and there exists a homeomorphism $\phi$ from 
$\Ch (\B)$ onto $\Ch (\A)$ and a clopen subset $K$ of $\Ch (\B)$ 
such that the equality
\begin{equation*}
T_Ef(y)=T1(y) \times
\begin{cases}
f(\phi (y)), & y\in K \\
\overline{f(\phi (y))}, & y\in \Ch (\B)\setminus K
\end{cases}
\end{equation*}
holds for every $f\in \A$. In particular, $\frac{T_E}{T1}$ is a real-algebra 
isomorphism. Thus we see that the equality
\[
\|T_EfT_Eg-1\|_{\infty (Y)}=\|fg-1\|_{\infty (X)}
\]
holds for every pair $f$ and $g$ in $\A$.
\end{theorem}

\begin{proof}
Since 
\[
\|T1T1-1\|_{\infty(Y)}=\|1^2-1\|_{\infty(X)}=0,
\] 
we see that $(T1)^2=1$. 
First we consider the case where $T1=1$. By Lemma \ref{0} we see that 
$T$ is injective and $\|TfTg\|_{\infty(Y)}=\|fg\|_{\infty(X)}$ 
for every $f,g \in \A^{-1}$. 
Then by Theorem \ref{0korovkin} there exists a homeomorphism $\phi$ 
from $\Ch (\B)$ onto $\Ch (\A)$ such that $|Tf(y)|=|f(\phi (y))|$ holds for 
every $f\in \A^{-1}$ and every $y\in \Ch (\B)$. 
We show that $Tf(y)=f(\phi (y))$ if $y\in K$ and 
$Tf(y) = \overline{f(\phi (y))}$ if $y \in \Ch (\B)\setminus K$, 
where $K$ is defined as in Definition \ref{K}. 

Suppose that $y\in K$ and $Tf(y)\ne f(\phi (y))$. 
By Lemma \ref{bishop} there exists an $H\in P_\B^0(y)$ such that 
$\sigma_{\pi}(TfH)=\{Tf(y)\}$. 
Since $Ti=i$ on $K$ and $Ti=-i$ on $\Ch (\B)\setminus K$ we see that 
the closure $\bar K$ of $K$ in $Y$ and the closure 
$\overline{\Ch (\B) \setminus K}$ of $\Ch (\B) \setminus K$ are disjoint. 
Thus we may assume that
 $|TfH|<|Tf(y)|$ on $\Ch (\B)\setminus K$. Put $h=T^{-1}H$. Then 
 by Lemma \ref{8} we see that $h\in P_\A^0(\phi (y))$. 
 Put $\alpha = \frac{-\overline{f(\phi (y))}}{|f(\phi (y))|}$. 
Then we have that 
\[
\|T\alpha TfTh-1\|_{\infty(Y)}<|Tf(y)|+1.
\]
(Suppose not; $\|T\alpha TfTh-1\|_{\infty(Y)}\ge |Tf(y)|+1$. Since 
\[
\|T\alpha TfTh-1\|_{\infty(Y)}\le \|T\alpha\|_{\infty(Y)}
\|TfTh\|_{\infty(Y)}+1=
|Tf(y)|+1,
\]
we have that 
\[
\|T\alpha TfTh-1\|_{\infty(Y)}=|Tf(y)|+1,
\]
so 
there exists a 
$z\in \Ch (\B)$ such that 
\[
|(T\alpha TfTh)(z)-1|=|Tf(y)|+ 1.
\]
Since $|T\alpha |=1$ on $\Ch (\B)$, we have that
\begin{multline*}
\|T\alpha TfTh\|_{\infty(Y)}
=\|T\alpha TfTh\|_{\infty (\Ch (\B))}=
\|TfTh\|_{\infty (\Ch (\B))} \\
=\|TfTh\|_{\infty(Y)}=|Tf(y)|,
\end{multline*}
thus 
\begin{equation}\label{000}
(T\alpha TfTh)(z)=-|Tf(y)|.
\end{equation}
Then we see that $z$ is a point in $K$.
Suppose that $z\in \Ch (\B) \setminus K$. Then by the definition of $H$, 
$|(TfTh)(z)|<|Tf(y)|$ holds. Since $|T\alpha|=1$ we have that 
\[
|(T\alpha TfTh)(z)-1|\le |(T\alpha TfTh)(z)|+1<|Tf(y)|+1,
\]
which 
is a contradiction proving $z\in K$. 
So $T\alpha (z)=\alpha$  by Lemma \ref{7} 
and thus we have by the equation 
(\ref{000}) that
\[
\frac{-\overline{f(\phi (y))}}{|f(\phi (y))|}Tf(z)Th(z)=-|Tf(y)|.
\] 
So $|Tf(z)Th(z)|=|Tf(y)|$. Since $\sigma_{\pi}(TfTh)=\{Tf(y)\}$, 
we see that $Tf(z)Th(z)=Tf(y)$, so that 
$\frac{\overline{f(\phi (y))}}{|f(\phi (y))|}Tf(y)=|Tf(y)|$. 
Then $\overline{f(\phi (y))}Tf(y)=|Tf(y)|^2$ holds since 
$|Tf(y)|=|f(\phi (y))|$, so that $f(\phi (y))=Tf(y)$, 
which is a contradiction.) Then by Lemma \ref{8} there exists 
an $h'\in P_\A^0(\phi (y))$ such that 
$T(\alpha h')=(T\alpha)(Th)$. Thus 
\begin{multline*}
|Tf(y)|+1> \|T\alpha TfTh-1\|_{\infty(Y)}\\
=\|T(\alpha h')Tf-1\|_{\infty(Y)}
=\|\alpha h'f-1\|_{\infty(X)} \\
\ge 
|\alpha h'(\phi (y))f(\phi (y))-1|=
|f(\phi (y))|+1=|Tf(y)|+1,
\end{multline*}
which is a contradiction proving that $Tf(y)=f(\phi (y))$. 

Suppose that $y\in \Ch (\B)\setminus K$. We show that $Tf(y)=
\overline{f(\phi (y))}$. 
Suppose not. By Lemma \ref{bishop} and Lemma \ref{8} 
there exists an $h \in P_\A^0(\phi (y))$ 
such that $\sigma_{\pi}(TfTh)=\{Tf(y)\}$. We may assume that 
$|TfTh|<|Tf(y)|$ on $K$ as the same way as in the case that 
$y\in K$. Put $\alpha =
\frac{-\overline{f(\phi (y))}}{|f(\phi (y))|}$. Then 
as in the same way as in the case where $y\in K$, we see that 
\[
\|T\alpha TfTh-1\|_{\infty(Y)}<|Tf(y)|+1.
\]
By Lemma \ref{8} there exists an $h'
\in P_\A^0(\phi (y))$ such that $T(\alpha h')=T(\alpha)Th$. 
Then
we have that
\begin{multline*}
\|T\alpha TfTh-1\|_{\infty(Y)}=\|T(\alpha h')Tf-1\|_{\infty(Y)}=
\|\alpha h'f-1\|_{\infty(X)} \\
\ge
|\alpha h'(\phi (y))f(\phi (y))-1|=|f(\phi (y))|+1
=|Tf(y)|+1,
\end{multline*}
which is a contradiction proving that 
$Tf(y)=\overline{f(\phi (y))}$. 
So we see that the equality 
\begin{equation*}
Tf(y)=
\begin{cases}
f(\phi (y)), & y\in K, \\
\overline{f(\phi (y))}, & y\in \Ch (\B)\setminus K
\end{cases}
\end{equation*}
holds for every $f\in \A^{-1}$ if $T1=1$. 

We consider the general case. Put $\tilde T=\frac{T}{T1}$. 
By a simple calculation we see that 
$\tilde T$ is a map from $\A^{-1}$ onto $\B^{-1}$ such that 
\[
\|\tilde Tf\tilde Tg-1\|_{\infty(Y)}=\|fg-1\|_{\infty(X)}
\]
holds for every $f,g \in \A^{-1}$ and 
$\tilde T1=1$. Then by the first part of the proof 
we see that there is a homeomorphism $\phi$ from 
$\Ch (\B)$ onto $\Ch (\A)$ and a clopen subset 
$K$ of $\Ch (\B)$ such that the equality
\begin{equation*}
\frac{Tf(y)}{T1(y)}=\tilde T(f)(y)=
\begin{cases}
f(\phi (y)), & y\in K \\
\overline{f(\phi (y))}, & y\in \Ch (\B)\setminus K.
\end{cases}
\end{equation*}
holds for every $f\in \A^{-1}$. 
Let $I_{\A}$ (resp. $I_{\B}$) denote the map $I_{\A}f=f|_{\Ch (\A)}$ 
(resp. $I_{\B}f=f|_{\Ch (\B)}$) from $\A$ (resp. $\B$) onto 
$\A|_{\Ch (\A)}$ (resp. $B|_{\Ch (\B)}$). Then $I_{\A}$ (resp. 
$I_{\B}$) is an algebra isomorphism from $\A$ (resp. $\B$) onto 
$A|_{\Ch (\A)}$ (resp. $B|_{\Ch (\B)}$). Put
a map $T_0$ from $A|_{\Ch (\A)}$ into $B|_{\Ch (\B)}$ by 
\begin{equation*}
T_0(f|_{\Ch (\A)})(y)=
T1(y) \times 
\begin{cases}
f(\phi (y)), & y\in K, \\
\overline{f(\phi (y))}, & y\in \Ch (\B)\setminus K
\end{cases}
\end{equation*}
for $f|_{\Ch (\A)} \in \A|_{\Ch (\A)}$. Put 
\[
T_E=I_{\B}^{-1}\circ T_0\circ I_{\A}.
\]
Then it is easy to see that $T_E$ is a bijection from 
$\A$ onto $\B$ which is an extension of $T$, and $\frac{T_E}{T1}$ is a 
real-algebra isomorphism from $\A$ onto $\B$ which is an extension of 
$\frac{T}{T1}$.  By the definition  we see that the equality
\begin{equation*}
T_Ef(y)=T1(y)\times
\begin{cases}
f(\phi (y)), & y\in K, \\
\overline{f(\phi (y))}, & y\in \Ch (\B) \setminus K
\end{cases}
\end{equation*}
holds for every $f\in \A$. 
\end{proof}

Let $A$ be a unital commutative Banach algebra. 
Recall that 
the spectral radius of $f\in A$ is denoted 
by $\rr (f)$. Then by the definition 
\[
\rr (f)=\|\hat f\|_{\infty (M_A)}
\]
holds for every $f\in A$, where $\hat f$ denotes the Gelfand transformation 
in $A$.

Recall that $\mcl A$ is the uniform closure of the Gelfand transform of $A$ 
in $C(M_A)$. If $A$ is semisimple, then we may suppose that 
$A\subset \mcl A$. Recall also that the Gelfand transform of 
$f\in A$ is denoted also by $f$ omitting $\hat{\cdot}$ for 
simplicity.

\begin{cor}\label{semisimplemain}
Let $A$ be a unital semisimple commutative Banach algebra and $B$ a unital 
commutative Banach algebra. Suppose that $T$ is a map from $A^{-1}$ onto 
$B^{-1}$ which satisfies that the equation 
\[
\rr (TfTg-1)=\rr (fg-1)
\]
holds 
for every pair $f$ and $g$ in $A^{-1}$. 
Then $B$ is semisimple and $(T1)^2=1$, and $T$ is 
extended to a map $T_E$ from 
$\mcl A$ onto $\mcl B$ such that $T_EA=B$, and 
there exist a homeomorphism $\phi$ from 
$\Ch (\mcl B)$ onto $\Ch (\mcl A)$ and a clopen subset $K$ of $\Ch (\mcl B)$ 
which satisfies that the equation
\begin{equation*}
T_Ef(y)= T1(y) \times 
\begin{cases}
f(\phi (y)), & y\in K \\
\overline{f(\phi (y))}, & y\in \Ch (\mcl B)\setminus K
\end{cases}
\end{equation*}
holds for every $f\in \mcl A$. 
In particular, $\frac{T_E}{T1}$  is a real-algebra isomorphism from 
$\mcl A$ onto $\mcl B$ and $\frac{T_E}{T1}(A)=B$; $A$ is 
real-algebraically isomorphic to $B$. Thus the equality
\[
\|T_EfT_Eg-1\|_{\infty (M_B)}=\|fg-1\|_{\infty (M_A)}
\]
holds for every pair $f$ and $g$ in $\mcl A$ and $T_E((\mcl A)^{-1})=
(\mcl B)^{-1}$. 
\end{cor}

\begin{proof}
First we consider the case where $B$ is semisimple. Let $f\in A^{-1}$. Since 
\[
\rr (TfTf^{-1}-1) 
= \rr (ff^{-1}-1)=0
\]
and $B$ is assumed to be semisimple, we have that 
$Tf^{-1}=(Tf)^{-1}$. So 
\[
\rr(\frac{Tf}{Tg}-1)=\rr (\frac{f}{g}-1)
\]
holds for every pair $f$ and $g$ in $A^{-1}$. 
(Recall that we denote the Gelfand transform of $f\in A$ also by $f$; omitting 
$\hat \cdot$, and we suppose that 
\[
A\subset \mcl A \subset C(M_A),\quad B\subset \mcl B\subset C(M_B),
\]
so that such a formula like $\frac{f}{g}$ is well-defined and so on.) 
We show that $T$ is extended to a map 
from $(\mcl A)^{-1}$ 
onto $(\mcl B)^{-1}$. Note that $M_A=M_{\mcl A}$ and so 
$(\mcl A)\cap C(M_A)^{-1}
=(\mcl A)^{-1}$. Let $f \in (\mcl A)^{-1}$. Then there exists a sequence 
$\{f_n\}$ in $A^{-1}$ such that $\|f-f_n\|_{\infty (M_A)}\to 0$ 
as $n\to \infty$. Since $f$ is invertible 
there exists a positive number $M$ such that 
$\frac{1}{M}<|f_n|<M$ holds on $M_A$ 
for every $f_n$. 
Thus 
\[
\frac{1}{M}|f_n(x)-f_m(x)|\le |\frac{f_n(x)}{f_m(x)}-1|
\le M|f_n(x)-f_m(x)|
\]
hold for every $x\in M_A$, so 
we have that 
\[
\frac{1}{M}\rr (f_n-f_m)\le \rr (\frac{Tf_n}{Tf_m}-1) 
\le M\rr (f_n-f_m)
\]
since $\rr (\frac{Tf_n}{Tf_m}-1)=\rr (\frac{f_n}{f_m}-1)$. 
Since $\rr (T1T1-1)=\rr (1^2-1)=0$ we have that $(T1)^2=1$ and so we have that 
\begin{multline*}
\rr (Tf_m)= \rr (Tf_mT1)\le \rr (Tf_mT1-1)+1\\
=\rr (f_m-1)+1\le M+2.
\end{multline*}
It follows that 
\[
\rr (Tf_n-Tf_m)\le \rr (Tf_m)\rr (\frac{Tf_n}{Tf_m}-1)\le M(M+2)\rr (f_n-f_m),
\]
so that $\{Tf_n\}$ is a Cauchy sequence in $B^{-1}$ with respect to the 
supremum norm on $M_B$. Thus there is an 
$F\in C(M_B)$ with 
\[
\|Tf_n-F\|_{\infty (M_B)}\to 0
\]
as $n\to \infty$. In a way similar to the above 
we see that $\rr (\frac{1}{Tf_m})\le M+2$ and so that 
$\frac{1}{M+2} \le |Tf_m|$ holds on $M_B$. Thus we see that $\frac{1}{M+2}
\le |F|$ on $M_B$. It follows that $F$ is invertible in $C(M_B)$ and so 
$F\in (\mcl B)^{-1}$ since $M_{\mcl B}=M_B$. 
In a routine argument we see that for each $f\in (\mcl A)^{-1}$ $F$ is 
uniquely determined; it does not depend on the choice of the 
sequence $\{f_n\}$ which converges to $f$. 
We define a map $T_e$ from $(\mcl A)^{-1}$ into $(\mcl B)^{-1}$ by $T_ef=F$. 
We show that $T_e$ is a surjection. Suppose that $F\in (\mcl B)^{-1}$. Then 
there 
is a sequence $\{F_n\}$ in $B$ such that $\|F_n-F\|_{\infty (M_B)}\to 0$ as $n 
\to \infty$. Since $F\in (\mcl B)^{-1}$ and $M_{\mcl B}=
M_B$, we may assume that $F_n\in B^{-1}$ for every $n$. 
Since $TA^{-1}=B^{-1}$, there exists a sequence $\{f_n\}$ in $A^{-1}$ 
with $Tf_n=F_n$ for every positive integer $n$.  As in a way 
similar to the above
we see $\{f_n\}$ is a Cauchy sequence which uniformly converges to some $f\in 
(\mcl A)^{-1}$. By the definition of $T_e$ we have that $T_ef=F$; $T_e$ is a 
surjection. Suppose that $f,g\in (\mcl A)^{-1}$. Then there are 
some $\{f_n\}$ and $\{g_n\}$ in $A^{-1}$ such that $f_n\to f$ and 
$g_n \to g$ as $n\to \infty$. Then 
we see that 
\[
\rr (Tf_nTg_n-1)=\rr (f_ng_n-1)
\]
and letting 
$n\to \infty$ we have that
\[
\|T_efT_eg-1\|_{\infty (M_B)}=\|fg-1\|_{\infty (M_A)}.
\] 
Applying Theorem \ref{main} to the map $T_e$ from $(\mcl A)^{-1}$ onto 
$(\mcl B)^{-1}$, we see that there exists a 
clopen subset $K$ of $\Ch (\mcl B)$ and $T_e$ is extended to a map 
$T_E$ from $\mcl A$ onto $\mcl B$ such that the equality
\begin{equation*}
T_Ef(y)=T1(y)\times
\begin{cases}
f(\phi (y)),& y\in K, \\
\overline{f(\phi (y))}, & y\in \Ch (\mcl B)\setminus K
\end{cases}
\end{equation*}
holds for every $f\in \mcl A$. It follows that $\frac{T_E}{T1}$ 
is a real-algebra isomorphism from $\mcl A$ onto $\mcl B$. 
We also see that $\frac{T_E}{T1}(A)=B$. (Since $T_E=T$ on $A^{-1}$ and 
$T1\in B^{-1}$ we see that $\frac{T_E}{T1}(A^{-1})=B^{-1}$. 
Suppose that $F\in B$. Then there exist an $F_0\in B^{-1}$ and a complex
number 
$\lambda$ with $F=F_0+\lambda$. If $\lambda =0$, then $F=F_0$ and 
there exists an $f_0\in A^{-1}$ with $Tf_0=T1F_0$, so that
$\frac{T_E}{T1}(f_0)=F_0=F$. If $\lambda\ne 0$, then there exist 
an $f_0\in A^{-1}$ and $f_{\lambda} \in A^{-1}$ with 
$Tf_0=T1F_0$ and $Tf_{\lambda}=\lambda T1$. Since $T_E$ is real-linear, 
we have that 
$T_E(f_0+f_{\lambda})=T1F$. We see that 
$\frac{T_E}{T1}(A)=B$.)

Finally we consider the general case; $B$ is not assumed to be semisimple. 
Let $\Gamma$ denote the Gelfand transform on $B$. Then $\Gamma \circ T$ is a 
function from $A^{-1}$ onto $(\hat B)^{-1}$, where $\hat B$ is the 
Gelfand transform of $B$. Then by the first part of the proof we see that 
there exist a homeomorphism $\phi$ from 
$\Ch (\mcl \hat B)$ onto $\Ch (\mcl A)$ 
and a clopen subset $K$ of $\Ch (\mcl \hat B)$ such that the equality
\begin{equation*}
\Gamma \circ Tf(y)=\Gamma \circ T1(y)\times 
\begin{cases}
f(\phi(y)), & y\in K, \\
\overline{f(\phi (y))}, & y\in \Ch (\mcl \hat B)\setminus K.
\end{cases}
\end{equation*}
holds for every $f\in A^{-1}$. 
In particular, we see that $\Gamma \circ T$ is an injection from $A^{-1}$ onto 
$(\hat B)^{-1}$, so $\Gamma$ is injective on $B^{-1}$ and so $\Gamma$ is injective 
on $B$ by a simple calculation. Thus we see that $B$ is semisimple. 
Then applying the first case we have the conclusion. 
\end{proof}

\begin{cor}\label{cor3.2}
Let $A$ be a unital semisimple commutative Banach algebra and $B$ a unital 
commutative Banach algebra. Suppose that $T$ is a map from $A^{-1}$ onto 
$B^{-1}$ which satisfies that the equality 
\[
\rr (TfTg-1)=\rr (fg-1)
\]
holds 
for every pair $f$ and $g$ in $A^{-1}$. 
Suppose that there exists a $\lambda \in {\mathbb C}\setminus 
\left({\mathbb R}\cup i{\mathbb R}\right)$ 
such that $T\lambda =\lambda$. Then $B$ is 
semisimple, $T$ is extended to a complex-algebra isomorphism $T_E$ from 
$\mcl A$ onto $\mcl B$ such that $T_EA=B$ and 
$T_E((\mcl A)^{-1})=(\mcl B)^{-1}$, and 
there exists a homeomorphism $\phi$ from $\Ch (\mcl B)$ 
onto $\Ch (\mcl A)$ such that the equality 
\[
T_Ef(y)=f(\phi (y)), \quad y\in \Ch (\mcl B)
\]
holds for every $f\in \mcl A$. 
\end{cor}

\begin{proof}
By Corollary \ref{semisimplemain} and its proof, we see that 
$B$ is semisimple and $(T1)^2=1$, and $T$ is 
extended to a map $T_E$ from $\mcl A$ onto $\mcl B$ 
such that there exists a homeomorphism $\phi$ from 
$\Ch (\mcl B)$ onto $\Ch (\mcl A)$ and a clopen subset $K$ of $\Ch (\mcl B)$ 
which satisfies that the equation
\begin{equation*}
T_Ef(y)= T1(y) \times 
\begin{cases}
f(\phi (y)), & y\in K \\
\overline{f(\phi (y))}, & y\in \Ch (\mcl B)\setminus K
\end{cases}
\end{equation*}
holds for every $f\in \mcl A$. In particular, $\frac{T_E}{T1}$ 
is a real-algebra 
isomorphism from $\mcl A$ onto $\mcl B$ such that 
$\frac{T_E}{T1}(A)=B$. 
We show that $K=\Ch (\mcl B)$. Suppose not. Then there is a $y\in \Ch (\mcl B)
\setminus K$. So by the hypothesis, 
\[
\lambda = T\lambda (y) = T1(y)\overline{\lambda (\phi (y))}
=\bar{\lambda}T1(y), 
\]
which is a contradiction since $T1(y)=1$ or $-1$ and 
$\lambda \in {\mathbb C}\setminus 
\left({\mathbb R}\cup i{\mathbb R}\right)$. Then we also see that 
$T1=1$ since 
\[
\lambda=T\lambda =T_E\lambda =T1\lambda
\]
on $\Ch(\mcl B)$ 
\end{proof}

\begin{cor}\label{cor3.3}
Let $A$ and $B$ 
be unital semisimple commutative Banach algebras. 
Suppose that $T$ is a group homomorphism from $A^{-1}$ onto $B^{-1}$. 
Then the following are equivalent.

{\rm (i)} $T$ is isometry with respect to the spectral radius, that is, 
the equality 
\[
\rr (Tf-Tg)=\rr (f-g)
\]
holds 
for every pair $f$ and $g$ in $A^{-1}$. 

{\rm (ii)} The equality 
\[
\rr (Tf-1)=\rr (f-1)
\]
holds for every $f\in A^{-1}$. 

{\rm (iii)} There exist a homeomorphism from $\Ch (\mcl B)$ onto 
$\Ch (\mcl A)$ and a clopen subset $K$ of $\Ch (\mcl B)$ such that 
the equality
\begin{equation*}
Tf(y)=
\begin{cases}
f(\phi (y)), & y\in K \\
\overline{f(\phi (y))}, & y\in \Ch (\mcl B)\setminus K
\end{cases}
\end{equation*}
holds for every $f\in A^{-1}$.

\noindent
Thus if one of the above holds, then $T$ is extended to a real-algebra 
isomorphism from $A$ onto $B$; $A$ is real-algebraically isomorphic to $B$. 
\end{cor}

\begin{proof}
The inclusion (iii) $\to$ (i) is trivial. Since $T1=1$, (i) implies (ii) is 
also trivial. By Corollary \ref{semisimplemain} we see that (ii) implies 
(iii). Thus $T$ is extended to a real-algebra isomorphism from 
$A$ onto $B$ if at least one of (i), (ii) and (iii) holds.
\end{proof}

The group isomorphism in Corollary \ref{cor3.3}   is not 
extended to the {\it complex} algebra isomorphism in general, as the following 
example shows.

\begin{example}
Let $A(\bar D)$ be the disk algebra on the closed unit disk; the algebra of 
all complex-valued continuous functions on the closed unit disk $\bar D$ 
which are analytic on the interior of the disk, and
$\overline{A(\bar D)}=\{f\in C(\bar D)|\bar f\in A(\bar D)\}$. 
Put ${\mathcal A}
=A(\bar D)\oplus {A(\bar D)}$, the direct sum of two copies of $A(\bar D)$ 
and 
${\mathcal B}
=A(\bar D)\oplus \overline{A(\bar D)}$, the direct sum of $A(\bar D)$ 
and $\overline{A(\bar D)}$. 
Then ${\mathcal A}$ and ${\mathcal B}$ are uniform algebras on 
$X=\bar D \times \{1,2\}$. 
Then $T(f\oplus g)=f\oplus \bar g$ defined on ${\mathcal A}^{-1}$ 
is a group isomorphisms onto ${\mathcal B}^{-1}$ such that 
\[
\|T(f\oplus g)-1\|_{\infty (X)}=\|f\oplus  g-1\|_{\infty (X)}. 
\]
On the other hand 
${\mathcal A}$ is not algebraically isomorphic to ${\mathcal B}$ 
as a complex algebra.
\end{example} 

\begin{cor}\label{cor3.4}
Let $A$ and $B $ be unital semisimple commutative Banach algebras. 
Suppose that $T$ is a group homomorphism 
 from $A^{-1}$ onto 
$B^{-1}$ which satisfies that the equality 
\[
\rr (Tf-1)=\rr (f-1)
\]
holds 
for every $f \in A^{-1}$. 
Suppose that there exists a $\lambda \in {\mathbb C}\setminus 
{\mathbb R}$ such that $T\lambda =\lambda$. Then $T$ is extended to a 
complex-algebra isomorphism $T_E$ from $\mcl A$ onto $\mcl B$ with $T_EA=B$, 
and 
there exists a homeomorphism $\phi$ from $\Ch (\mcl B)$ 
onto $\Ch (\mcl A)$ such that the equality 
\[
T_Ef(y)=f(\phi (y)), \quad y\in \Ch (\mcl B)
\]
holds for every $f\in \mcl A$. 
\end{cor}

\begin{proof}
By Corollary \ref{cor3.3} we see that 
there exist a homeomorphism from $\Ch (\mcl B)$ onto 
$\Ch (\mcl A)$ and a clopen subset $K$ of $\Ch (\mcl B)$ such that 
the equality
\begin{equation*}
Tf(y)=
\begin{cases}
f(\phi (y)), & y\in K \\
\overline{f(\phi (y))}, & y\in \Ch (\mcl B)\setminus K
\end{cases}
\end{equation*}
holds for every $f\in A^{-1}$. We also see that 
$K=\Ch (\mcl B)$. 
Suppose that $y\in \Ch (\mcl B)\setminus K$, then 
we have that 
\[
\lambda=T\lambda(y)=\overline{\lambda (\phi (y))}=
\bar{\lambda},
\]
which is a contradiction since $\lambda 
\in {\mathbb C}\setminus {\mathbb R}$. Thus $K=\Ch (\mcl B)$. 
As in the same way as the proof of Theorem \ref{main} $T$ is 
extended to the desired $T_E$.
\end{proof}

\begin{cor}\label{cor3.5}
Let $A$ 
be a unital semisimple commutative Banach algebra and 
$B$ a unital commutative Banach algebra. 
Suppose that $T$ is a map from $A^{-1}$ onto 
$B^{-1}$ and $\alpha$ is a non-zero complex number 
which satisfies that the equality 
\[
\rr (TfTg-\alpha)=\rr (fg-\alpha)
\]
holds 
for every pair $f$ and $g$ in $A^{-1}$. 
Then $B$ is 
semisimple, and $T$ is extended to a map $T_E$ from 
$\mcl A$ onto $\mcl B$ such that $T_EA=B$, 
and there exist an element $\eta \in B^{-1}$ such that 
$\eta ^{2}=1$, 
a homeomorphism $\phi$ from $\Ch (\mcl B)$ 
onto $\Ch (\mcl A)$ and a clopen subset $K$ of $\Ch (\mcl B)$ 
which satisfies that the equality 
\begin{equation*}
T_Ef(y)=\eta (y) \times 
\begin{cases}
f(\phi (y)), & y\in K \\
\frac{\alpha}{|\alpha |}\overline{f(\phi (y))}, & y\in 
\Ch (\mcl B)\setminus K
\end{cases}
\end{equation*}
holds for every $f\in \mcl A$. 
Furthermore, if $T1=1$, then 
the equality 
\begin{equation*}
T_Ef(y)=
\begin{cases}
f(\phi (y)), & y\in K, \\
\overline{f(\phi (y))}, & y \in \Ch (\mcl B)\setminus K
\end{cases}
\end{equation*}
holds for every $f\in \mcl A$. 
Thus the equality 
\[
\|T_EfT_Eg-\alpha \|_{\infty (M_B)}=\|fg-\alpha \|_{\infty (M_A)}
\]
holds for every pair $f$ and $g$ in $\mcl A$, and 
$T_E((\mcl A)^{-1})=(\mcl B)^{-1}$ holds. 
Furthermore, if $T1=1$ and $\alpha \in {\mathbb C}\setminus {\mathbb R}$ or 
there exists $\lambda \in {\mathbb C}\setminus {\mathbb R}$ such that 
$T\lambda =\lambda$, then the equality 
\[
T_Ef(y)=f(\phi (y)), \quad y\in \Ch (\mcl B)
\]
holds for every $f\in \mcl A$. In this case, $T_E$ is extended to a 
complex-algebra isomorphism from $\mcl A$ onto $\mcl B$ with 
$T_EA=B$. 
\end{cor}

\begin{proof}
Let $\beta$ be a complex number with $\beta^2=\alpha$. Put a function 
$T_{\beta}$ from $A^{-1}$ into $B^{-1}$ by 
\[
T_{\beta}f=\frac{1}{\beta}
T({\beta}f) \quad f\in A^{-1}.
\]
Then by a simple calculation $T_{\beta}A^{-1}=B^{-1}$.
Since $T_{\beta}fT_{\beta}g=\frac{1}{\alpha}T(\beta f)T(\beta g)$, we have 
that 
\begin{multline*}
\rr (T_{\beta}fT_{\beta}g-1)=\frac{1}{|\alpha|}\rr (T(\beta f)T(\beta g)-
\alpha ) \\
= \frac{1}{|\alpha|}\rr (\beta f\beta g- 
\alpha )=\rr (fg-1).
\end{multline*}
Then by Corollary \ref{semisimplemain}, $B$ is semisimple and 
$(T_{\beta}1)^2=1$. 
We also see by Corollary \ref{semisimplemain} that $T_{\beta}$ is 
extended to a function $(T_{\beta})_E$ from $\mcl A$ onto 
$\mcl B$ and that there exists a 
homeomorphism $\phi$ from $\Ch (\mcl B)$ onto $\Ch (\mcl A)$ and a clopen 
subset of $\Ch (\mcl B)$ such that the equation
\begin{equation*}
(T_{\beta})_Ef(y)=T_{\beta}1(y) \times 
\begin{cases}
f(\phi (y)), & y\in K, \\
\overline{f(\phi (y))}, & y\in \Ch (\mcl B)\setminus K
\end{cases}
\end{equation*}
holds for every $f\in \mcl A$. 
Put $\eta = T_{\beta}1$ and put 
\[
T_Ef=\beta (T_{\beta})_E(\frac{f}{\beta})
\]
for $f\in \mcl A$. By a simple calculation we see that $T_E$ is an 
extension of $T$ which maps from $\mcl A$ onto $\mcl B$ and we have that
\begin{equation*}
T_Ef(y)= \eta (y) \times 
\begin{cases}
f(\phi (y)), & y\in K, \\
\frac{\alpha}{|\alpha|}\overline{f(\phi (y))}, & y \in \Ch (\mcl B)
\setminus K
\end{cases}
\end{equation*}
since $\frac{\beta}{\bar \beta}=\frac{\alpha}{|\alpha|}$.

Suppose that $T1=1$. Then by the above equation we have that 
$1=\eta $ on $K$ and $1=\frac{\alpha}{|\alpha|}\eta $ on 
$\Ch (\mcl B)\setminus K$. Thus we see that the equality 
\begin{equation}\label{te}
T_Ef(y)= 
\begin{cases}
f(\phi (y)), & y\in K, \\
\overline{f(\phi (y))}, & y \in \Ch (\mcl B)\setminus K
\end{cases}
\end{equation}
holds for every $f\in \mcl A$. 

Furthermore suppose that $T1=1$, and $\alpha \in {\mathbb C}\setminus
{\mathbb R}$ or $T\lambda =\lambda$ for some 
$\lambda \in {\mathbb C}\setminus {\mathbb R}$. 
First we consider the case where 
$\alpha \in {\mathbb C}\setminus
{\mathbb R}$.
Since $B$ is semisimple and $T1=1$ we have that
\[
\rr (T\alpha -\alpha)=\rr (T\alpha T1-\alpha)=\rr (\alpha 1-\alpha)=0,
\]
so $T\alpha=\alpha$ on $\Ch (\mcl B)$. On the other hand 
by putting $f=\alpha$ in the above equation (\ref{te}) we have that
\begin{equation*}
T_E\alpha (y) =
\begin{cases}
\alpha, & y \in K, \\
|\alpha|, & y\in \Ch (\mcl B)\setminus K.
\end{cases}
\end{equation*}
Since $\alpha \in {\mathbb C}\setminus {\mathbb R}$ we see that 
$K=\Ch (\mcl B)$ and so the equality 
\[
T_Ef(y)=f(\phi (y)), \quad y \in \Ch (\mcl B)
\]
holds for every $f\in \mcl A$. 
Next we consider the case where 
$T\lambda =\lambda$ for some 
$\lambda \in {\mathbb C}\setminus {\mathbb R}$. 
Then by the above equation (\ref{te}) we have that 
\begin{equation*}
T\lambda (y)=T_E\lambda (y) =
\begin{cases}
\lambda, & y\in K, \\
\bar \lambda, & y\in \Ch (\mcl B)\setminus K. 
\end{cases}
\end{equation*}
It follows that $K=\Ch (\mcl B)$ since $T\lambda =\lambda$. So we have that 
\[
T_Ef(y)=f(\phi (y)), \quad y \in \Ch (\mcl B)
\]
holds for every $f\in \mcl A$.
\end{proof}

\begin{cor}\label{cor3.6}
Let $A$ be a unital semisimple commutative Banach algebra and 
$B$ a unital commutative Banach algebra. 
Suppose that $T$ is a group homomorphism 
 from $A^{-1}$ onto 
$B^{-1}$ which satisfies that there exists a nonzero 
complex number $\alpha$ such that 
the equality 
\[
\rr (Tf-\alpha)=\rr (f-\alpha)
\]
holds 
for every $f \in A^{-1}$. 
Then $B$ is semisimple, $T$ is extended to a real-algebra isomorphism 
$T_E$ from $\mcl A$ onto $\mcl B$ with $T_EA=B$ and 
$T_E((\mcl A)^{-1})=(\mcl B)^{-1}$,  and there exists a homeomorphism $\phi$ 
from $\Ch (\mcl B)$ onto $\Ch (\mcl A)$ and a clopen subset $K$ of 
$\Ch (\mcl B)$ such that 
\begin{equation*}
T_Ef(y)= 
\begin{cases}
f(\phi (y)), & y\in K, \\
\overline{f(\phi (y))}, & y\in \Ch (\mcl B)\setminus K
\end{cases}
\end{equation*}
holds for every $f\in \mcl A$. 
Furthermore if $\alpha \in {\mathbb C}\setminus {\mathbb R}$ or 
there exists a $\lambda \in {\mathbb C}\setminus {\mathbb R}$ 
such that $T\lambda = \lambda$, then 
\[
T_Ef(y)=f(\phi (y)), \quad y\in \Ch (\mcl B)
\]
holds for every $f\in \mcl A$; In this case $T_E$ is a 
complex-algebra isomorphism. 
\end{cor}

\begin{proof}
Since $T$ is a group homomorphism we see that
\[
\rr (TfTg-\alpha)= \rr(fg-\alpha)
\]
holds for every $f$ and $g$ in $A^{-1}$. 
Then $B$ is semisimple by Corollary \ref{cor3.5}. 
We also see by 
Corollary 
\ref{cor3.5} that 
$T$ is extended to a real-algebra isomorphism 
$T_E$ from $\mcl A$ onto $\mcl B$ with $T_EA=B$, and 
there exist a homeomorphism $\phi$ from 
$\Ch (\mcl B)$ onto $\Ch (\mcl A)$ and a clopen subset 
$K$ of $\Ch (\mcl B)$ which satisfies that the equality 
\begin{equation*}
T_Ef(y)=
\begin{cases}
f(\phi (y)), &y\in K, \\
\overline{f(\phi (y))},& y\in \Ch (\mcl B)\setminus K
\end{cases}
\end{equation*}
holds for every $f\in \mcl A$. 
Note that $T1=1$ since $T$ is a group homomorphism from 
$A^{-1}$ onto $B^{-1}$. 

Suppose that $\alpha \in {\mathbb C}\setminus {\mathbb R}$ or 
there exists a $\lambda \in {\mathbb C}\setminus 
{\mathbb R}$ such that $T\lambda=\lambda$. 
Then by Corollary \ref{cor3.5} 
we see that the equality
\[
T_Ef(y)=f(\phi (y)), \quad y\in \Ch (\mcl B)
\]
holds for every $f\in \mcl A$. 
\end{proof}

\section{Non-symmetric multiplicatively spectrum and peripheral 
spectrum-preserving maps between invertible groups}
In this section we consider non-symmetric multiplicatively (peripheral) 
spectrum-preserving maps. We say that 
a map $T$ from the invertible group $A^{-1}$ 
of a unital commutative Banach algebra $A$ into the invertible group
$B^{-1}$ of a unital commutative Banach algebra $B$ is non-symmetric 
multiplicatively (resp. peripheral) spectrum-preserving 
if there exists a nonzero complex 
number $\alpha$ such that 
\[
\sigma (TfTg-\alpha )=\sigma (fg-\alpha )
\]
\[
(\text{resp.} \sigma_{\pi}(TfTg-\alpha )= \sigma _{\pi}(fg-\alpha))
\]
holds for every pair $f$ and $g$ in $A^{-1}$. In this section we show, 
under some additional assumption, that non-symmetric multiplicatively 
peripheral spectrum-preserving maps from $A^{-1}$ onto $B^{-1}$ are 
extended to algebra isomorphisms from $A$ onto $B$. 
\begin{cor}\label{cor3.7}
Let $A$ be a unital semisimple commutative Banach algebra and $B$ a unital 
commutative Banach algebra. Suppose that $T$ is a map from $A^{-1}$ onto 
$B^{-1}$ which satisfies that 
there exists a nonzero complex number $\alpha$ such that  
\[
\sigma_{\pi}(TfTg-\alpha) \cap
\sigma_{\pi}(fg-\alpha) \ne \emptyset
\]
holds for every pair $f$ and $g$ in $A^{-1}$.
Then $B$ is semisimple, and $T$ is extended to a map $T_E$ from $\mcl A$ onto 
$\mcl B$ with $T_EA=B$,  and 
there exists  a homeomorphism 
$\phi$ from $\Ch (\mcl B)$ onto $\Ch (\mcl A)$ such that the equality 
\[
T_Ef(y)=T1(y)f(\phi (y)), \quad y\in \Ch (\mcl B)
\]
holds for every $f\in \mcl A$. In particular, $\frac{T}{T1}$ is 
extended to a 
complex-algebra isomorphism $\frac{T_E}{T1}$ from $\mcl A$ onto $\mcl B$ 
with $\frac{T_E}{T1}(A)=B$. 
\end{cor}

\begin{proof}
First we consider the case where $\alpha=1$. We have that the equality
\[
\rr (TfTg-1)=\rr (fg-1), \quad f,g\in A^{-1}
\]
holds since 
\[
\sigma_{\pi}(TfTg-1)\cap \sigma_{\pi}(fg-1)\ne
\emptyset
\]
holds for 
every pair $f$ and $g$ in $A^{-1}$. Thus by Corollary \ref{semisimplemain}
we see that $B$ is semisimple,  and 
$T$ is extended to the map $T_E$ from $\mcl A$ onto $\mcl B$, and 
there exist a homeomorphism 
$\phi$ from $\Ch (\mcl B)$ onto $\Ch (\mcl A)$ and a clopen subset 
$K$ of $\Ch (\mcl B)$ such that the equality
\begin{equation*}
T_Ef(y) = T1(y)\times 
\begin{cases}
f(\phi (y)), \quad& y\in K, \\
\overline{f(\phi (y))}, \quad &y \in \Ch(\mcl B)\setminus K
\end{cases}
\end{equation*}
holds for every $f\in \mcl A$, where $(T1)^2=1$. 
Let $\lambda \in \sigma_{\pi}(f)$ for an $f\in A^{-1}$ (resp. $B^{-1}$). 
Then there exists an $x\in M_A$ (resp. $M_B$) with $f(x)=\lambda$. 
Since $|\lambda|=\|f\|_{\infty (X)}$ (resp. $|\lambda|=\|f\|_{\infty (Y)}$), 
we have that $f^{-1}(\lambda)$ is a peak set for $\mcl A$ (resp. $\mcl B$). 
Thus there exists an $x_0\in \Ch (\mcl A)\cap f^{-1}(\lambda)$ 
(resp. $x_0\in \Ch (\mcl B)\cap f^{-1}(\lambda)$) by Corollary 
2.4.6 in \cite{B}. It follows that 
$\lambda =f(x_0)\in f(\Ch (\mcl A))$ (resp. 
$\lambda =f(x_0)\in f(\Ch (\mcl B))$), so that 
$\sigma_{\pi}(f) \subset f(\Ch (\mcl A))$ (resp. 
$\sigma_{\pi}(f) \subset f(\Ch (\mcl B))$) holds for every 
$f\in A^{-1}$ (resp. $B^{-1}$).

We show that $K=\Ch (\mcl B)$. Suppose not; There exists a $y_0 \in 
\Ch (\mcl B)\setminus K$. By the definition of $K$ we have that 
$\bar K \cap \overline{\Ch (\mcl B)
\setminus K}=\emptyset$, 
where $\bar \cdot$ denotes the closure in $M_B$. Thus there exists a 
$U\in B^{-1}$ such that $U(y_0)=i$, $|U|<\frac13$ on $K$, and 
$\mathrm{Im}U>0$ on $\Ch (\mcl B)$. Then there exists a $u\in A^{-1}$ with 
$Tu=T_Eu=\frac{U}{T1}$ since $TA^{-1}=B^{-1}$, so 
\begin{equation*}
U(y)=Tu(y)T1(y)=
\begin{cases}
u\circ \phi (y), & y\in K \\
\overline{u\circ \phi (y)}, & y\in \Ch (\mcl B)\setminus K
\end{cases}
\end{equation*}
since $(T1)^2=1$. Thus 
\begin{multline*}
\sigma_{\pi}(TuT1-1)=\sigma_{\pi}(U-1) \\
\subset U(\Ch (\mcl B))-1
\subset
\{z\in {\mathbb C}:\mathrm{Im}z>0\}.
\end{multline*}
Since 
\[
i=U(y_0)=
\overline{u\circ \phi (y_0)}, 
\]
we have $u(\phi (y_0))=-i$. Since $U=\overline{u\circ \phi}$ on 
$\Ch (\mcl B)\setminus K$, we have
$\mathrm{Im}u<0$ on $\phi (\Ch (\mcl B)\setminus K)$. 
Since $U=u\circ \phi$ on $K$, we have that 
$|u|<\frac13$ on $\phi (K)$. 
Since $(u-1)(\phi (y_0))=-i-1$, we have 
$$\sigma_{\pi}(u-1)\subset  \{z\in {\mathbb C}:|z|\ge \sqrt 2\}.$$
Thus $(u-1)(\phi (K))\cap \sigma_{\pi}(u-1)=\emptyset$. 
It follows that
\[
\sigma_{\pi}(u-1)\subset \{z\in {\mathbb  C}: \mathrm{Im}z<0\}.
\]
So we see that
\[
\sigma_{\pi}(TuT1-1)\cap \sigma_{\pi}(u-1)=\emptyset,
\]
which is a contradiction proving that $K=\Ch (\mcl B)$. 
It follows that the equality
\[
T_Ef(y)=T1(y)f(\phi (y))
\]
holds for every 
$f\in \mcl A$ and $y\in \Ch (\mcl B)$.

Finally we consider the general case for $\alpha$. 
Let $\beta$ be a complex number with $\beta ^2=\alpha$. 
Put a function $T_{\beta}$ from $A^{-1}$ onto $B^{-1}$ by 
$T_{\beta}f=\frac{1}{\beta}T({\beta}f)$ for $f\in A^{-1}$. 
Then $T_{\beta}$ is well-defined and we see by a simple calculation that 
$T_{\beta}A^{-1}=B^{-1}$. Then we see that the equalities 
\begin{multline*}
\sigma_{\pi}(T_{\beta}fT_{\beta}g-1) = 
\sigma_{\pi}(\frac{1}{\alpha}T(\beta f)T(\beta g)-1) \\
= 
\frac{1}{\alpha}\sigma_{\pi}(T(\beta f)T(\beta g)-\alpha)
=\sigma_{\pi}(fg-1)
\end{multline*}
hold for every pair $f$ and $g$ in $A^{-1}$. Thus by the first part of the 
proof we see that $B$ is semisimple, 
and $T_{\beta}$ is extended to a map $(T_{\beta})_E$ from
$\mcl A$ onto $\mcl B$, and there exists a homeomorphism 
$\phi$ from $\Ch (\mcl B)$ onto $\Ch (\mcl A)$ such that the equality 
\[
(T_{\beta})_Ef(y)=T_{\beta}1(y)f(\phi (y)), \quad y\in \Ch (\mcl B)
\]
holds for every $f\in \mcl A$. Put  $T_E:\mcl A\to \mcl B$ by 
$T_Ef=\beta (T_{\beta})_E(\frac{f}{\beta})$ for $f\in \mcl A$. Then 
since we see that the equality 
\[
T_Ef(y)={\beta}(T_{\beta})_E(\frac{f}{\beta})(y)=
T_{\beta}1 (y) f(\phi (y)),  \quad y\in \Ch (\mcl B)
\]
holds for every $f\in \mcl A$. 
Then we have 
\[
T1(y)=T_E1(y)=T_{\beta}1 (y),\quad y\in \Ch (\mcl B)
\]
holds. Thus we see that $T_{\beta}1 =T1$ and the conclusion holds. 
\end{proof}

\begin{cor}\label{cor3.8}
Let $A$ be a unital semisimple commutative Banach algebra and $B$ a unital 
commutative Banach algebra. Suppose that $T$ is a group homomorphism 
from $A^{-1}$ onto 
$B^{-1}$ which satisfies that 
there exists a non-zero complex number $\alpha$ such that 
\[
\sigma_{\pi}(Tf-\alpha) \cap
\sigma_{\pi}(f-\alpha) \ne \emptyset
\]
holds for every $f$ in $A^{-1}$.
Then $B$ is semisimple, and $T$ is extended to a complex-algebra 
isomorphism $T_E$ from $\mcl A$ onto $\mcl B$ with $T_EA=B$,  and 
there exists a homeomorphism 
$\phi$ from $\Ch (\mcl B)$ onto $\Ch (\mcl A)$ such that the equality 
\[
T_Ef(y)=f(\phi (y)), \quad y\in \Ch (\mcl B)
\]
holds for every $f\in \mcl A$. 
\end{cor}

\begin{proof}
Since $T$ is a group homomorphism we see that
\[
\sigma_{\pi}(TfTg-\alpha)\cap \sigma_{\pi}(fg-\alpha)
=\sigma_{\pi}(T(fg)-\alpha)\cap \sigma_{\pi}(fg-\alpha)\ne
\emptyset
\]
holds for every pair $f$ and $g$ in $A^{-1}$. Thus by Corollary 
\ref{cor3.7} 
$T$ is extended to a map $T_E$ from $\mcl A$ onto $\mcl B$ with 
$T_EA=B$, and there exists 
a homeomorphism $\phi$ from $\Ch (\mcl B)$ onto $\Ch (\mcl A)$ 
such that the equality 
\[
T_Ef(y)=T1(y)f(\phi (y)), \quad y\in \Ch (\mcl B)
\]
holds for every $f\in \mcl A$. 
Since $T$ is a group homomorphism from $A^{-1}$ onto $B^{-1}$ we have 
that $T1=1$, so 
we conclude that the equation 
\[
T_Ef(y)=f(\phi (y)), \quad y\in \Ch (\mcl B)
\]
holds for every $f\in \mcl A$. 
\end{proof}

\section{Surjections between commutative Banach algebras}
Let $A$ and $B$ be unital semisimple commutative Banach algebras. In this 
section we consider the maps from $A$ onto $B$, which satisfy the similar 
conditions for maps from $A^{-1}$ onto $B^{-1}$ in the previous sections. 
Multiplicatively spectrum-preserving maps are 
initiated by Moln\'ar \cite{m1}, Rao and Roy \cite{rr1} and Hatori, 
Miura and Takagi \cite{hmt1} extended the results of Moln\'ar for 
uniform algebras. 
Luttman and Tonev \cite{lt} extended the results of 
Rao and Roy and Hatori, Miura and Takagi (for uniform algebras) 
 in the case where the maps 
between uniform algebras are 
multiplicatively peripheral spectrum-preserving. Inspired by the 
theorem of Luttman and Tonev we have considered the 
following question. 
\begin{quest}
Suppose that $\A$ and $\B$ are uniform algebras and $T$ is a map from $\A$ 
onto $\B$. Suppose that 
\[
\|TfTg+1\|_{\infty (Y)}=\|fg+1\|_{\infty (X)}
\]
holds for every pair $f$ and $g$ in $\A$ and $T\lambda =\lambda$ for 
every complex number $\lambda$. Does it follow that $T$ is an algebra 
isomorphism from 
$\A$ onto $\B$?
\end{quest}
In this section we give a complete solution to the above question in more 
general form (cf. Theorem \ref{semisimpleconjecture} and Corollary 
\ref{conjecture}). We also 
give a generalization of a theorem of Luttman and Tonev (cf. Corollary 
\ref{glt}.)
\begin{theorem}\label{0korovkinAB}
Let $\A$ and $\B$ be uniform algebras on compact Hausdorff spaces 
$X$ and $Y$ respectively and $S$ a map from $\A$ onto $\B$. 
Suppose that the equality $\|SfSg\|_{\infty (Y)}=
\|fg\|_{\infty (X)}$ holds for every pair $f$ and $g$ in $\A$. 
Then there exists a homeomorphism $\phi$ from $\Ch (\B)$ onto 
$\Ch (\A)$ such that the equality
\[
|Sf(y)|=|f(\phi (y)|, \quad y\in \Ch (\B)
\]
holds for every $f\in \A$. 
\end{theorem}
Note that $S$ need not be injective. 
\begin{proof}
We can prove Theorem \ref{0korovkinAB} in a way similar to the proof of 
Theorem \ref{0korovkin} and we only show a sketch of the proof. 

In a way similar in the proof of Theorem \ref{0korovkin} we see that
$|S1(y)|=1$ for every $y\in \Ch (\B)$. Thus we see that the equality 
$\|Sf\|_{\infty (Y)}=\|f\|_{\infty (X)}$ holds for every $f\in \A$. 
Let $y\in \Ch (\B)$ and put
\begin{multline*}
L_y=\{x\in X:\text{$|f(x)|=1$ for every $f\in \A$ with} \\ 
|Sf(y)|=1=\|Sf\|_{\infty (Y)}\}.
\end{multline*}
Then $L_y$ is a singleton whose element is a point in $\Ch (\A)$. Put 
a function $\phi$ from $\Ch (\B)$ into $\Ch (\A)$ by $\phi (y)=$ the 
unique element of $L_y$. Then in the same way as in the proof of Theorem 
\ref{0korovkin} we see that 
$|Sf(y)|=|f(\phi (y))|$ holds for every $f\in \A$ and $y\in \Ch(\B)$ 
if $Sf(y)\ne0$ and $f(\phi(y))\ne 0$. 
We show that $|Sf(y)|=|f(\phi (y))|$ holds even if $Sf(y)=0$ or 
$f(\phi (y))=0$. Suppose that $Sf(y)=0$. Then for every positive $\varepsilon$ 
there exists an $H_{\varepsilon}\in P_{\B}^0(y)$ such that 
$\|SfH_{\varepsilon}\|_{\infty (Y)}<\varepsilon$. Since $S$ is a 
surjection there is an $h_{\varepsilon}\in \A$ with 
$Sh_{\varepsilon}=H_{\varepsilon}$. Then by the definition of $\phi (y)$, 
$|h_{\varepsilon}(\phi (y))|=1$ holds since 
\[
Sh_{\varepsilon}(y)=H_{\varepsilon}(y)=1=
\|H_{\varepsilon}\|_{\infty (Y)}=\|Sh_{\varepsilon}\|_{\infty (Y)}.
\] 
Thus we have that 
\[
\varepsilon > \|SfH_{\varepsilon}\|_{\infty (Y)}
=
\|fh_{\varepsilon}\|_{\infty (X)}\ge
|f(\phi (y))h_{\varepsilon}(\phi (y))|=|f(\phi (y))|,
\]
so the equalities $f(\phi (y))=0=Sf(y)$ holds since $\varepsilon$ is 
arbitrary. Suppose that $f(\phi (y))=0$. Then for every positive 
$\varepsilon$, there exists a $u_{\varepsilon}\in P_{\A}^0(\phi (y))$ 
such that $\|fu_{\varepsilon}\|_{\infty (X)}<\varepsilon$. 
We see that $|Su_{\varepsilon}(y)|=1$. Suppose not. Then 
$|Su_{\varepsilon}(y)|<1$ since $\|Su_{\varepsilon}\|_{\infty (Y)}
=\|u_{\varepsilon}\|_{\infty (X)}=1$. So there exists an $H\in 
P_{\B}^0(y)$ with $\|Su_{\varepsilon}H\|_{\infty (Y)}<1$. 
Since $S$ is a surjection there is an $h\in \A$ with 
$Sh=H$. Then by the definition of $\phi(y)$ we see that 
$|h(\phi (y)|=1$ since $Sh(y)=1=\|Sh\|_{\infty (Y)}$. 
It follows that 
\[
1>\|Su_{\varepsilon}H\|_{\infty (Y)}=\|u_{\varepsilon}h\|_{\infty (X)}
\ge |u_{\varepsilon}(\phi (y))h(\phi (y))|=|u_{\varepsilon}(\phi (y))|,
\]
which is a contradiction proving that $|Su_{\varepsilon} (y)|=1$. Then 
we have that
\[
\varepsilon >\|fu_{\varepsilon}\|_{\infty (X)}=
\|SfSu_{\varepsilon}\|_{\infty (Y)}\ge |Sf(y)Su_{\varepsilon}(y)|
=|Sf(y)|,
\]
and so we have that $Sf(y)=0=f(\phi (y))$ since $\varepsilon$ is 
arbitrary. We conclude that 
the equality 
\[
|Sf(y)|=|f(\phi (y))|,\quad y\in \Ch (\B)
\]
holds for every $f\in \A$. In the same way as in the proof of 
Theorem \ref{0korovkin} we see that $\phi$ is continuous. 

Next let $x\in \Ch (\A)$ and put
\begin{multline*}
K_x=\{y\in Y:\text{$|Sf(y)|=1$ for every $f\in \A$ with} \\ 
|f(x)|=1=\|f\|_{\infty (X)}\}.
\end{multline*}
In a way similar in the proof of Theorem \ref{0korovkin} we see that 
$K_x$ is a singleton which consists of a point in $\Ch (\B)$. Put a function 
$\psi$ from $\Ch (\A)$ into $\Ch (\B)$ such that $\psi (x)=$ the unique 
element in $K_x$. In a way similar in the proof of Theorem 
\ref{0korovkin} and the first part of the proof we see that the 
equality
\[
|Sf(\psi (x))|=|f(x)|, \quad x\in \Ch (\A)
\]
holds for every $f\in \A$ and $\psi$ is continuous on $\Ch (\A)$. 
Again in a way similar in the proof of Theorem \ref{0korovkin} we 
see that $\phi\circ\psi$ and $\psi\circ\phi$ are identity functions 
on $\Ch (\A)$ and $\Ch (\B)$ respectively, so that $\phi$ is a homeomorphism 
from $\Ch (\B)$ onto $\Ch (\A)$. 
\end{proof}

The following is a generalization of a theorem of 
Luttman and Tonev \cite{lt} and it is related to Corollary 3 in 
\cite{l}. 

\begin{cor}\label{glt}
Let $\A$ and $\B$ be uniform algebras on compact Hausdorff spaces 
$X$ and $Y$ respectively. Suppose that $S$ is a map from 
$\A$ onto $\B$ such that the inclusion
\[
\sigma_{\pi}(SfSg)\subset \sigma_{\pi}(fg)
\]
holds for every pair $f$ and $g$ in $\A$. Then 
$(S1)^2=1$ and there exists a homeomorphism $\phi$ from $\Ch (\B)$ onto 
$\Ch (\A)$ such that the equality 
\[
Sf(y)=S1(y)f(\phi (y)), \quad y\in \Ch (\B)
\]
holds for every $f\in \A$. In particular, $\frac{S}{S1}$ is an isometrical 
algebra isomorphism from $\A$ onto $\B$. 
\end{cor}

\begin{proof}
A proof is similar to that of Corollary \ref{vlt} and we sketch a proof. 
First we consider the case where $S1=1$. By the inclusion 
$\sigma_{\pi}(SfSg)\subset \sigma_{\pi}(fg)$, we have that 
$\|SfSg\|_{\infty (Y)}=\|fg\|_{\infty (X)}$ holds for every pair 
$f$ and $g$ in $\A$. Thus by Theorem \ref{0korovkinAB} 
there exists a homeomorphism $\phi$ from 
$\Ch (\B)$ onto $\Ch (\A)$ such that the equality 
\begin{equation}\label{=}
|Sf(y)|=|f(\phi (y))|, \quad y\in \Ch (\B)
\end{equation}
holds for every $f\in \A$. We show that the equality 
\[
Sf(y)=f(\phi (y)), \quad y\in \Ch (\B)
\]
holds for every $f\in \A$. 
If $f(\phi(y))\ne 0$, then the proof is similar to that in the proof of 
Corollary \ref{vlt}. 
If $f(\phi (y))=0$, then $Sf(y)=f(\phi (y))$ holds 
by the above equation \ref{=}.

Finally we consider the general case; We do not assume $S1=1$. 
In a way similar to the proof of Corollary \ref{vlt} we see that 
$(S1)^2=1$. So $S1\in \B^{-1}$.
Put a map $\tilde S$ from $\A$ into $\B$ by $\tilde Sf
=\frac{Sf}{S1}$. Then $\tilde S$ is a surjection and the inclusion
\[
\sigma_{\pi}(\tilde Sf\tilde Sg)\subset \sigma_{\pi}(fg)
\]
holds for every pair $f$ and $g$ in $\A$ since $(S1)^2=1$. 
By the first part of the proof we see that the conclusion holds.
\end{proof}

\begin{theorem}\label{semisimpleconjecture}
Let $A$ be a unital semisimple commutative Banach algebra and $B$ 
a unital commutative Banach algebra. Suppose that $S$ is a map from 
$A$ onto $B$ 
which satisfies that there exists a non-zero complex number $\alpha$ 
such that the equality
\[
\rr (SfSg-\alpha )=\rr (fg-\alpha )
\]
holds for every pair $f$ and $g$ in $A$. Then $B$ is semisimple and 
there exist an $\eta\in B$ with $\eta ^2=1$, a homeomorphism $\phi$ 
from $\Ch (\mcl B)$ onto 
$\Ch (\mcl A)$ and a clopen subset $K$ of $\Ch (\mcl B)$ such that the 
equality 
\begin{equation*}
Sf(y)=\eta (y)\times
\begin{cases}
f(\phi (y)), & y\in K, \\
\frac{\alpha}{|\alpha|}\overline{f(\phi (y))}, 
& y \in \Ch (\mcl B)\setminus K
\end{cases}
\end{equation*}
holds for every $f\in A$. Furthermore if $S1=1$, then $\eta=1$. 
Furthermore if $S1=1$, and $\alpha \in {\mathbb C}\setminus{\mathbb R}$ or 
there exists a $\lambda \in {\mathbb C}\setminus {\mathbb R}$ such that 
$S\lambda = \lambda$, then $K=\Ch (\mcl B)$ and the equality 
\[
Sf(y)=f(\phi (y)), \quad y\in \Ch (\mcl B)
\]
holds for every $f\in A$. 
\end{theorem}

\begin{proof}
First we consider the case where $B$ is semi-simple. We show that 
$SA^{-1}=B^{-1}$. Let $f\in A^{-1}$. Put $g=\alpha f^{-1}$. Then we have that 
\[
0=\rr (fg-\alpha) = \rr (SfSg-\alpha ), 
\]
so $SfSg=\alpha$ for $B$ is semisimple. Since $\alpha$ is a non-zero 
complex number we see that $Sf\in B^{-1}$; 
we have proved that $SA^{-1}\subset B^{-1}$. Suppose that $F\in B^{-1}$. 
Since $SA=B$, there exist an $f$ and a $g$ in $A$ with $Sf=F$ and 
$Sg= \alpha F^{-1}$. Then we have that 
\[
0=\rr (SfSg-\alpha )=\rr (fg - \alpha ),
\]
so $fg=\alpha$ for $A$ is semisimple. Thus we see that $f\in A^{-1}$. 
It follows that $SA^{-1}=B^{-1}$. Applying Corollary \ref{cor3.5} to 
$S|A^{-1}$ we see that there corresponds the extended map $(S|A^{-1})_E$ from 
$\mcl A$ onto $\mcl B$ such that the equality
\begin{equation}\label{7.0}
\|(S|A^{-1})_Ef(S|A^{-1})_Eg-\alpha \|_{\infty (M_B)} 
=\|fg-\alpha \|_{\infty (M_A)}
\end{equation}
holds for every pair $f$ and $g$ in $\mcl A$. 
We also see by Corollary \ref{cor3.5} that there exist an $\eta\in B^{-1}$ 
with $\eta ^2=1$, a homeomorphism $\phi$ from $\Ch (\mcl B)$ onto 
$\Ch (\mcl A)$ and a clopen subset $K$ of $\Ch (\mcl B)$ which satisfy that 
\begin{equation*}
((S|A^{-1})_E(f))(y)= \eta (y) \times 
\begin{cases}
f(\phi (y)), & y\in K, \\
\frac{\alpha}{|\alpha|}\overline{f(\phi (y))},& y\in \Ch (\mcl B)\setminus K.
\end{cases}
\end{equation*}
If $S1=1$, and $\alpha \in {\mathbb C}\setminus{\mathbb R}$ or 
there exists a $\lambda \in {\mathbb C}\setminus {\mathbb R}$ with 
$S\lambda = \lambda$,
then
\[
((S|A^{-1})_E(f))(y)=f(\phi (y)), \quad y\in \Ch (\mcl B)
\]
holds for every $f\in A$ by Corollary \ref{cor3.5}.

We show that $(S|A^{-1})_E=S$ on $A$. 
By the definition of 
$(S|A^{-1})_E$ we have that for every $g\in (\mcl A)^{-1}$ 
there exists a sequence $\{g_n\}$ in $A^{-1}$ with 
$\|g-g_n\|_{\infty (M_A)}\to 0$ as $n\to \infty$, and 
the 
equality
\[
\|(S|A^{-1})_Eg-Sg_n\|_{\infty (M_B)}\to 0
\]
holds as $n\to \infty$. Thus for every 
$f\in A$ and $g\in (\mcl A)^{-1}$ the equality 
\[
\|SfSg_n-\alpha\|_{\infty (M_B)}=\|fg_n-\alpha\|_{\infty (M_A)}
\]
holds, where
$\{g_n\}\subset A^{-1}$ and $\|g-g_n\|_{\infty (M_A)} \to 0$ 
as $n\to \infty$.
Letting $n\to \infty$ we see that the equality
\begin{equation}\label{7.1}
\|Sf(S|A^{-1})_Eg-\alpha\|_{\infty (M_B)}=\|fg-\alpha \|_{\infty (M_A)}
\end{equation}
holds for every $f\in A$ and $g\in (\mcl A)^{-1}$. Since 
$(S|A^{-1})_E((\mcl A)^{-1})=(\mcl B)^{-1}$ holds  
we see that the equality
\begin{equation*}
\|SfG-\alpha\|_{\infty (M_B)}=\|(S|A^{-1})_EfG-\alpha\|_{\infty (M_B)}
\end{equation*}
holds for every $f\in A$ and $G\in (\mcl B)^{-1}$ by the equations 
(\ref{7.0}) and (\ref{7.1}). Applying 
peaking function argument as before it follows that the equality
\[
Sf(y)=(S|A^{-1})_Ef(y), \quad y\in \Ch (\mcl B)
\]
holds for every $f\in A$. 
We show a proof for a convenience. 
Substituting $G$ by $nG$ for positive integer $n$, we have that
\[
\|Sf(nG)-\alpha\|_{\infty (M_B)}=\|(S|A^{-1})_Ef(nG)-\alpha\|_{\infty (M_B)}, 
\]
so
\[
\|SfG-\frac{\alpha}{n}\|_{\infty (M_B)}=
\|(S|A^{-1})_EfG-\frac{\alpha}{n}\|_{\infty (M_B)}, 
\]
and letting $n\to \infty$ we have
\[
\|SfG\|_{\infty (M_B)}=\|(S|A^{-1})_EfG\|_{\infty (M_B)}
\]
holds for every pair $f\in A$ and $G\in (\mcl B)^{-1}$. 
Suppose that $f\in A$ and $y\in \Ch (\mcl B)$. If 
$Sf(y)=0$, then there exists a sequence $\{G_n\}$ in $P_{\mcl B}^0(y)$ 
with 
\[
\|SfG_n\|_{\infty (M_B)}\to 0,
\]
as $n\to \infty$ so that
\[
\|(S|A^{-1})_EfG_n\|_{\infty (M_B)}\to 0.
\]
as $n\to \infty$. 
It follows that $(S|A^{-1})_Ef(y)=0$. In the same way we see that $Sf(y)=0$ 
if $(S|A^{-1})_Ef(y)=0$. Suppose that $Sf(y)\ne 0$ and 
$(S|A^{-1})_Ef(y)\ne 0$. Applying Lemma \ref{bishop} we see that 
there exists a $G\in P_{\mcl B}^0(y)$ such that 
\begin{equation}\label{2}
\sigma_{\pi}(SfG)=\{Sf(y)\}, \quad \sigma_{\pi}
((S|A^{-1})_EfG)=\{(S|A^{-1})_Ef(y)\}.
\end{equation}
Thus we see that 
\begin{equation}\label{3}
|Sf(y)|=\|SfG\|_{\infty (M_B)}=
\|(S|A^{-1})_EfG\|_{\infty (M_B)}
=|(S|A^{-1})_Ef(y)|.
\end{equation}
Put $\mu= \frac{-\alpha \overline{Sf(y)}}{|Sf(y)|}$. 
Then we have
\[
\|\mu Sf G-\alpha\|_{\infty (M_B)} = |\alpha|
\|\frac{\overline{Sf(y)}}{|Sf(y)|}SfG+1\|_{\infty (M_B)}=
|\alpha |\left(|Sf(y)|+1\right).
\]
On the other hand 
\begin{multline*}
\|\mu SfG - \alpha\|_{\infty (M_B)}= 
\|\mu (S|A^{-1})_EfG-\alpha\|_{\infty (M_B)}\\
=|\alpha |\|\frac{\overline{Sf(y)}}{|Sf(y)|}(S|A^{-1})_EfG+1\|_{\infty (M_B)}.
\end{multline*}
Applying the equations (\ref{2}) and (\ref{3}) we see that 
$Sf(y)=(S|A^{-1})_Ef(y)$. 
Thus we see that $S=(S|A^{-1})_E$ on $A$. 

Finally we consider the general case, where we do not assume that $B$ is 
semisimple. Let $\Gamma$ be the Gelfand transform on $B$. By applying the 
conclusion of the first part of the proof, we see that the map 
$\Gamma \circ S$ from $A$ onto $\hat B$, the Gelfand transform of $B$, is 
injective. It follows that $\Gamma$ is injective. Thus we see that 
$B$ is semisimple. 
Applying the first part of the proof 
we see that the conclusion holds. 
\end{proof}

Since uniform algebras are unital semisimple commutative 
Banach algebras, 
we see that the following holds.

\begin{cor}\label{conjecture}
Let $\A$ and $\B$ be uniform algebras on compact Hausdorff spaces 
$X$ and $Y$ respectively. Suppose that $S$ is a map from 
$\A$ onto $\B$ which satisfies that there exists a nonzero complex 
number $\alpha$ such that the equality
\[
\|SfSg-\alpha \|_{\infty (Y)}=\|fg-\alpha\|_{\infty (X)}
\]
holds for every pair $f$ and $g$ in $A$. Then there exist an $\eta \in B$ 
with $\eta ^2=1$, a homeomorphism $\phi$ from $\Ch (\B)$ onto $\Ch (\A)$, 
and a clopen subset $K$ of $\Ch (\B)$ such that the equality
\begin{equation*}
Sf(y)=\eta (y)\times
\begin{cases}
f(\phi (y)), &y\in K, \\
\frac{\alpha}{|\alpha |}\overline{f(\phi (y))}, & y\in \Ch (\B)\setminus K
\end{cases}
\end{equation*}
holds for every $f\in \A$. Furthermore if $T1=1$, then the equality
\begin{equation*}
Sf(y)=
\begin{cases}
f(\phi (y)),& y\in K, \\
\overline{f(\phi (y))}. & y\in \Ch (\B)\setminus K
\end{cases}
\end{equation*}
holds for every $f\in \A$. Furthermore if $T1=1$, and $\alpha \in {\mathbb C}
\setminus {\mathbb R}$ or there exists a $\lambda \in 
{\mathbb C}\setminus {\mathbb R}$ with $T\lambda = \lambda$, then the 
equality
\[
Tf(y)=f(\phi (y)), \quad y\in \Ch (\B)
\]
holds for every $f\in \A$. 
\end{cor}

\begin{cor}\label{spec}
Let $A$ be a unital semisimple commutative Banach algebra and 
$B$ a unital commutative Banach algebra. Suppose that $S$ is a map from 
$A$ onto $B$ which satisfies that there exists a non-zero complex number 
$\alpha$ such that 
\[
\sigma_{\pi}(SfSg-\alpha)\cap \sigma_{\pi}(fg-\alpha)\ne \emptyset
\]
holds for every pair $f$ and $g$ in $A$. Then $B$ is semisimple and 
there exists a homeomorphism $\phi$ from $\Ch (\mcl B)$ onto $\Ch (\mcl A)$ 
such that the equality 
\[
Sf(y)=S1(y)f(\phi (y)),\quad y \in \Ch (\mcl B)
\]
holds for every $f\in A$. 
\end{cor}

\begin{proof}
Since
\[
\sigma_{\pi}(SfSg-\alpha)\cap \sigma_{\pi}(fg-\alpha)\ne \emptyset, 
\]
holds for every $f,g\in A$, then
\[
\rr(SfSg-\alpha)=\rr(fg-\alpha)
\]
holds for every $f,g\in A$. Then by Theorem \ref{semisimpleconjecture} 
we see that 
$B$ is semisimple and 
$S$ is real-linear. 
On the other hand we see that 
$SA^{-1}=B^{-1}$. Suppose that $f\in A^{-1}$. Then 
\[
\sigma_{\pi}(f(\alpha f^{-1})-\alpha )=\{0\},
\]
so
\[
\sigma_{\pi}(SfS(\alpha f^{-1})-\alpha )\supset \{0\}.
\]
It follows that 
\[
\sigma_{\pi}(SfS(\alpha f^{-1})-\alpha )= \{0\}.
\]
We see that $Sf\in B^{-1}$. Thus we see that $SA^{-1}\subset B^{-1}$. 
Suppose conversely that $F\in B^{-1}$. then there exist an $f$ and a 
$g$ in $A$ with $Sf=F$ and $Sg=\alpha F^{-1}$. 
Then 
\[
\sigma_{\pi}(SfSg-\alpha)=\{0\},
\]
so
\[
\sigma_{\pi}(fg-\alpha)=\{0\}.
\]
We see that $f\in A^{-1}$. Thus we see that $B^{-1}\subset A^{-1}$, and 
thus $SA^{-1}=B^{-1}$. Then by Corollary \ref{cor3.7} 
there exists a homeomorphism $\phi$ from $\Ch (\mcl B)$ onto 
$\Ch (\mcl A)$ such that the equality
\begin{equation}\label{last}
Sf(y)=S1(y)f(\phi (y)),\quad y\in \Ch (\mcl B)
\end{equation}
holds for every $f\in A^{-1}$. We show that the equation 
(\ref{last}) holds for every $f\in A$. Let $f\in A$. Then there exist 
$f_0\in A^{-1}$ and a complex number $\mu$ with 
$f=f_0+\mu$. If $\mu=0$, then
\[
Sf(y)=Sf_0(y)=S1(y)f(\phi (y))=S1(y)f(\phi (y)), \quad y\in \Ch (\mcl B)
\]
holds. If $\mu \ne 0$, then, 
the by real-linearity of $S$ and 
the equation 
(\ref{last}) we have
\[
Sf(y)=Sf_0(y)+S\lambda (y)=S1(y)f_0(\phi (y))+S1(y)\lambda 
=S1(y)f(\phi (y))
\]
holds for every $y\in \Ch (\mcl B)$. 
\end{proof}

The authors do not know if a corresponding result for $\alpha =0$
holds.

\end{document}